\documentclass[10pt]{amsart}

\usepackage{amsmath, amsfonts, amssymb, amsthm, mathdots, mathrsfs, dsfont}

\usepackage{tikz-cd}

\usepackage[colorlinks, urlcolor=black]{hyperref}
\hypersetup{citecolor=blue, linkcolor=blue}

\newtheorem{theorem}{Theorem}[section]
\newtheorem{definition}[theorem]{Definition}
\newtheorem{proposition}[theorem]{Proposition}

\newtheorem{corollary}[theorem]{Corollary}
\newtheorem{lemma}[theorem]{Lemma}
\newtheorem{remark}[theorem]{Remark}

\newtheorem*{conjecture*}{Conjecture}

\numberwithin{equation}{section}

\author{Luis Alberto Lomel\'i}
\address{Luis Alberto Lomel\'i \\ Instituto de Matem\'aticas \\ Pontificia Universidad Cat\'olica de Valpara\'iso \\  Blanco Viel 596 \\ Cerro Bar\'on \\ Valpara\'iso \\ Chile}
\email{luis.lomeli@pucv.cl}

\begin{document}

\title[The Langlands-Shahidi method over function fields]{The Langlands-Shahidi method \\ over function fields}

\begin{abstract}
We establish the Langlands-Shahidi method over a global field of characteristic $p$. Furthermore, we prove that the Ramanujan Conjecture implies the Riemann Hypothesis for Langlands-Shahidi automorphic $L$-functions over function fields.
\end{abstract}

\maketitle

\section*{Introduction}

The Langlands-Shahidi method (LS method) offers a local-global procedure to construct, in a number of situations, a system of $\gamma$-factors, $L$-functions and $\varepsilon$-factors. Let $k$ be a global field with ring of ad\`eles $\mathbb{A}_k$. Consider:
\begin{enumerate}
   \item a pair $({\bf G},{\bf M})$ of quasi-split connected reductive groups such that $\bf M$ is a Levi component for a maximal parabolic subgroup $\bf P$ of $\bf G$ over $k$;
   \item an irreducible constituent $r$ of the adjoint action of ${}^LM$ on the Lie algebra of the unipotent radical of ${}^LP$;
   \item a cuspidal automorphic representation $\pi$ of ${\bf M}(\mathbb{A}_k)$ which is globally generic.
\end{enumerate}
We then have global $L$-functions and root numbers $L(s,\pi,r) \text{ and } \varepsilon(s,\pi,r), \ s \in \mathbb{C}$, which satisfy a functional equation
\begin{equation*}
   L(s,\pi,r) = \varepsilon(s,\pi,r) L(1-s,\tilde{\pi},r),
\end{equation*}
where $\tilde{\pi}$ denotes the contragredient representation.

The LS method was developed by Shahidi when $k$ is a number field in a series of papers with breakthrough work done in \cite{Sh1990}. Shahidi also treated the case of Archimedean fields \cite{Sh1985}. When $k$ is a function field, the author performed a study of the method for split classical groups under the assumption ${\rm char}(k) \neq 2$ in \cite{Lo2009,Lo2015}, and removed this restriction in \cite{Lo2019}. In this article, we present the Langlands-Shahidi method in general over a global field of characteristic $p$.

Locally, we obtain a system of $\gamma$-factors, $L$-functions and $\varepsilon$-factors at every place $v$ of $k$. Let $\psi : k \backslash \mathbb{A}_k \rightarrow \mathbb{C}^\times$ be a non-trivial character. Then we have
\begin{equation*}
   \gamma(s,\pi_v,r_v,\psi_v), \ L(s,\pi_v,r_v) 
   \text{ and } \varepsilon(s,\pi_v,r_v,\psi_v).
\end{equation*}
The definition is made with the aid of the Langlands-Shahidi local coefficient. For tempered representations, local factors satisfy the following relation
\begin{equation*}
   \varepsilon(s,\pi_v,r_v,\psi_v) = \gamma(s,\pi_v,r_v,\psi_v) \dfrac{L(s,\pi_v,r_v)}{L(1-s,\tilde{\pi}_v,r_v)}.
\end{equation*}

We begin in a purely local setting in Sections~\ref{LS:coeff} and~\ref{Weyl}, where we work over a non-archimedean local field $F$ of characteristic $p$. We take $\pi$ to be a generic representation of $M = {\bf M}(F)$ and $\psi$ a non-trivial character of $F$. The local coefficient
\begin{equation*}
   C_\psi(s,\pi,\tilde{w}_0)
\end{equation*}
is obtained via intertwining operators and the multiplicity one property of Whittaker models. Rank one cases are addressed in \S~\ref{rankone}. In particular, Proposition~\ref{rankoneprop} shows compatibility of the Langlands-Shahidi local coefficient with the abelian $\gamma$-factors of Tate's thesis \cite{T}. Then, Proposition~\ref{vary} determines the behavior of the local coefficient as one varies the system of Weyl group element representatives, Haar measures and additive character. Proposition~\ref{x:independence} addresses the well-definedness of local factors: two representations of $M$ that are conjugate by an element of the maximal torus of the adjoint group, and generic with respect to the same character, are isomorphic. The proof relies on results of Henniart and Vign\'eras on the restriction of representations \cite{HeVi}. In \S~\ref{Lgrp:LSrep} we define Langlands-Shahidi representations, and we gather important properties in \S~\ref{mult:ind}, including multiplicativity of the local coefficient and a useful lemma that appears in several of our inductive arguments. When $\pi$ has an Iwahori fixed vector, the local coefficient reduces to rank one cases and is expressed in terms of Artin $\gamma$-factors in \S~\ref{ups}.

A very useful local to global result, Lemma~\ref{hllemma}, allows us to lift any supercuspidal representation $\pi_0$ to a cuspidal automorphic representation $\pi$ with controlled ramification at all other places; if $\pi_0$ is generic, then the construction gives a globally generic $\pi$. Globally, there is a connection to Langlands' theory of Eisenstein series over function fields \cite{Ha1974,Mo1982}. This leads us to the crude functional equation involving the local coefficient and partial $L$-functions, in addition to rationality, summarized in \S~\ref{crude:rat}.

We then turn towards the main result of the LS method in characteristic $p$. Theorem~\ref{mainthm} establishes existence and uniqueness for a system of $\gamma$-factors, $L$-functions and root numbers. Local $\gamma$-factors are defined recursively by means of the local coefficient and they connect to the global theory via the functional equation
\begin{equation*}
   L^S(s,\pi,r) = \prod_{v \in S} \gamma(s,\pi_v,r_v,\psi_v) L^S(1-s,\tilde{\pi},r).
\end{equation*}
Partial $L$-functions being defined by
\begin{equation*}
   L^S(s,\pi,r) = \prod_{v \notin S} L(s,\pi_v,r_v).
\end{equation*}

An inspiring list of axioms for $\gamma$-factors that uniquely characterize them can be found in \cite{LaRa2005}. Work on the uniqueness of Rankin-Selberg $L$-functions for general linear groups \cite{HeLo2013a} led us to extend the characterization in a natural way to include $L$-functions and $\varepsilon$-factors beginning with the classical groups in \cite{Lo2015}. Lemma~\ref{hllemma} allows us to reduce the number of required axioms.

We observe that Axiom~(vii) in our list is known as Shahidi's tempered $L$-function conjecture; made in \cite{Sh1990}, proved in most cases by Kim \cite{Ki2005}, and settled in all cases of characteristic zero by Heiermann and Opdam \cite{HeOp2013}. Their method of proof is valid over any non-archimedean local field if one incorporates the $\gamma$-factors in characteristic $p$ provided by this article, together with Properties~(i)--(vi); this is the approach taken in \cite{dCaHeLo}, where Axiom~(vii) is established in characteristic $p$. Of independent interest is an alternative proof for split connected reductive groups via the Kazhdan transfer of Ganapathy \cite{Ga2015}, which can be found in \S~5 of \cite{Lo2019}. 

In Theorem~\ref{RamRie:thm}, we prove that the Ramanujan Conjecture implies the Riemann Hypothesis for LS automorphic $L$-functions over function fields, a result we obtain when we incorporate the results of L. Lafforgue for ${\rm GL}_n$ \cite{LaL} and those of V. Lafforgue for a connected reductive group scheme \cite{LaV}. Under the Ramanujan Conjecture, the resulting lift to ${\rm GL}_N$ is compatible with local factors at every place. We observe that the Ramanujan Conjecture is known for ${\rm GL}_n$ \cite{LaL} and split classical groups \cite{Lo2009,Lo2019}; for split exceptional groups see \cite{OpenAI2026}.

\subsection*{Acknowledgments} The author would like to thank professor F. Shahidi, with admiration for his mathematical work. He is very grateful to professor G. Henniart for many discussions that resulted over the years from collaborative work. Mathematical conversations with J. Bernstein and W. T. Gan at MSRI and MPIM were particularly insightful. The author would like to thank H. del Castillo, E. Goins, G. Harder, Y. Kim, M. Krishnamurthy, P. Kutzko, L. Lafforgue, V. Lafforgue, E. Lapid, C.~M\oe glin, J. Navarro, D. Prasad, A. Roche, R. Schmidt and M.-F. Vign\'eras. The author is very grateful to the following institutions: Mathematical Sciences Research Institute, Max-Planck-Institut f\"ur Mathematik and Institut des Hautes \'Etudes Scientifiques. Work on this article was supported in part by MSRI NSF Grant DMS 0932078, Project VRIEA-PUCV 039.367 and FONDECYT Grant 1212013.

\section{The Langlands-Shahidi local coefficient}\label{LS:coeff}

Let $F$ denote a non-archimedean local field of characteristic $p$. The ring of integers is denoted by $\mathcal{O}_F$ and a fixed uniformizer by $\varpi_F$. We let $q_F$ denote the cardinality of the residue field, and fix the absolute value whose value group is generated by $q_F$. Let $({\bf G},{\bf M})$ be a pair of quasi-split connected reductive group schemes, with $\bf M$ a maximal Levi subgroup of $\bf G$.

Given an algebraic group scheme $\bf H$, when the ring of definition is clear from context we let $H$ denote its group of rational points, e.g., $H = {\bf H}(F)$; we let ${\bf Z}_{\bf H}$ denote the center of $\bf H$.

\subsection{Local notation}\label{ss:local:not}
Let $\mathfrak{ls}(p,{\bf G},{\bf M})$, or simply $\mathfrak{ls}(p)$ when $\bf G$ and $\bf M$ are clear from context, denote the class of triples $(F,\pi,\psi)$ consisting of: a non-archimedean local field $F$, with ${\rm char}(F) = p$; a generic\footnote{We assume generic representations are irreducible.} representation $\pi$ of $M$; and, a smooth non-trivial additive character $\psi : F \rightarrow \mathbb{C}^\times$.

We say that $(F,\pi,\psi) \in \mathfrak{ls}(p)$ is supercuspidal, principal series, etc., if $\pi$ is such a generic representation.

Given a quasi-split connected reductive group scheme $\bf G$ we fix a Borel subgroup $\bf B$. Write ${\bf B} = {\bf T}{\bf U}$ where $\bf T$ is the maximal torus and $\bf U$ is the unipotent radical. This fixes a choice of simple roots $\Delta$ with respect to the maximal split subtorus ${\bf T}_s$ of $\bf T$. We only work with standard parabolic subgroups $\bf P$ of $\bf G$, i.e., we assume ${\bf P} \supset {\bf B}$; possible, since every parabolic is conjugate to a standard one.

Let ${\bf P}$ be a maximal parabolic subgroup of $\bf G$, with Levi component $\bf M$ and unipotent radical $\bf N$, so that ${\bf P} = {\bf M}{\bf N}$, and let $\mathfrak{n}$ denote the Lie algebra of $\bf N$. The group ${\bf M} / {\bf Z}_{\bf G}$ has a one dimensional split center, hence ${\rm Hom}_F({\bf M}/{\bf Z}_{\bf G},{\rm GL}_1)$ is a rank one free $\mathbb{Z}$-module. Let
\begin{equation*}
   \delta = \det({\rm Ad}_{\bf M} \vert_\mathfrak{n}) \in 
   {\rm Hom}_F({\bf M}/{\bf Z}_{\bf G},{\rm GL}_1).
\end{equation*}
We let $\alpha$ be the simple root such that the maximal parabolic ${\bf P}$ corresponds to the subset of simple roots $\theta = \Delta - \left\{ \alpha \right\}$. Let $\delta_{\alpha}$ be the unramified character of $M$ given by
\begin{equation}\label{deltatilde}
   \delta_\alpha: M \xrightarrow{\delta^l} {\rm GL}_1(F) 
   \xrightarrow{\left| \cdot \right|_F} \mathbb{C}^\times, \ 
   l = \left\langle \delta, \alpha^\vee \right\rangle^{-1},
\end{equation}
where $\left\langle \cdot\, , \cdot \right\rangle$ is the perfect pairing between roots and co-roots.

We make a few conventions concerning parabolic induction that we will use throughout the article. For this, we now let ${\bf P} = {\bf M}{\bf N}$ denote a parabolic subgroup of $\bf G$ that is not necessarily maximal. Let $(\pi,V)$, or simply $\pi$, be an irreducible smooth representation of $M$. By parabolic induction, we mean normalized parabolic induction
\begin{equation*}
   {\rm ind}_{P}^G (\pi),
\end{equation*}
where we extend the representation $\pi$ to $P = MN$ by making it trivial on $N$. Also, whenever the parabolic subgroup $\bf P$ and ambient group $\bf G$ are clear from context, we will simply write
\begin{equation*}
   {\rm Ind}(\pi) = {\rm ind}_{P}^G(\pi).
\end{equation*}
We also incorporate twists by an unramified character $\chi$, where we let
\begin{equation*}
   {\rm I}(\chi,\pi) = 
   {\rm ind}_P^G (\chi \otimes \pi)
\end{equation*}
and denote its corresponding space by ${\rm V}(\chi,\pi)$. More precisely, ${\rm V}(\chi,\pi)$ is the space of functions $f : G \rightarrow V$ such that: there exists a compact open subgroup $K$ of $G$ with $f(gk) = f(g)$ for every $g \in G$, $k \in K$; $f(mng) = \chi(m) \delta_P^{1/2}(m)\pi(m)f(g)$, for every $m \in M$, $n \in N$, $g \in G$, and $\delta_P$ is the modulus character of $P$. The representation ${\rm I}(\chi,\pi)$ is the right regular representation of $G$ on ${\rm V}(\chi,\pi)$.

Now, in the case of a maximal parabolic $\bf P$, we write
\begin{equation*}
   {\rm I}(s,\pi) = {\rm I}(\delta_{\!\alpha}^{s},\pi), \ s \in \mathbb{C},
\end{equation*}
with $\delta_\alpha$ as in equation~\eqref{deltatilde}; its corresponding space is denoted by ${\rm V}(s,\pi)$.

\subsection{The local coefficient and its rationality}\label{ss:LS:local:coef}
Let $\Sigma$ denote the roots of $\bf G$, relative to ${\bf T}_s$. Let $W$ denote the Weyl group of $\Sigma$, which is generated by simple reflections $w_\alpha$, $\alpha \in \Delta$. Standard parabolic subgroups of $\bf G$ are in correspondence with subsets $\theta \subset \Delta$; $\theta \leftrightarrow {\bf P} = {\bf P}_\theta$. We write ${\bf P}_\theta = {\bf M}_\theta{\bf N}_\theta$, and omit the subscripts when $\theta$ is clear from context. Let $W_\theta$ denote the subgroup of $W$ generated by $w_\alpha$, $\alpha \in \theta$. We set
\begin{equation} \label{w_0}
   w_0 = w_l w_{l,\theta},
\end{equation}
where $w_l$ and $w_{l,\theta}$ are the longest elements of $W$ and $W_\theta$, respectively. Given a Weyl group element $w$, we let $\tilde{w}$ denote a fixed choice of Weyl group element representative in the normalizer ${\rm N}(T_s)$. Specific choices of representatives and Haar measures on the relevant unipotent groups will be addressed in \S~\ref{Weyl} and will be in accordance with the semisimple rank one cases of \S~\ref{rankone}.

Given an irreducible smooth representation $\pi$ of $M_\theta$, let $\tilde{w}_0(\pi)$ be the representation given by $\tilde{w}_0(\pi)(x) = \pi(\tilde{w}_0^{-1}x\tilde{w}_0)$; and similarly for an unramified character $\chi$ and $\tilde{w}_0(\chi)$.
There is an intertwining operator for the representations ${\rm I}(\chi,\pi)$ and ${\rm I}(\tilde{w}_0(\chi),\tilde{w}_0(\pi))$,
\begin{equation*}
   {\rm A}(\chi,\pi,\tilde{w}_0) : {\rm V}(\chi,\pi) \rightarrow {\rm V}(\tilde{w}_0(\chi),\tilde{w}_0(\pi)),
\end{equation*}
defined via the principal value integral
\begin{equation*}
   {\rm A}(\chi,\pi,\tilde{w}_0)f(g) = \int_{N_{w_0}} f(\tilde{w}_0^{-1}ng) \, dn,
\end{equation*}
where $N_{w_0} = U \cap w_0 N_\theta^- w_0^{-1}$, with $N_\theta^-$ the unipotent radical of the parabolic subgroup opposite to $P_\theta$.

Given a non-degenerate character $\Psi : U \rightarrow \mathbb{C}^\times$, we let $\Psi_{\tilde{w}_0}$ be the non-degenerate character on the unipotent radical $U_\theta = M_\theta \cap U$ of $M_\theta$ defined by
\begin{equation}\label{nd:char}
   \Psi_{\tilde{w}_0}(u) = \Psi(\tilde{w}_0 u \tilde{w}_0^{-1}), \ u \in U_\theta.
\end{equation}
Whenever the above relationship holds, we say that $\Psi$ and $\Psi_{\tilde{w}_0}$ are $\tilde{w}_0$-compatible.

Fix a pair of $\tilde{w}_0$-compatible non-degenerate characters $\Psi$ and $\Psi_{\tilde{w}_0}$. Given an irreducible $\Psi_{\tilde{w}_0}$-generic representation $(\pi,V)$ of $M_\theta$, with Whittaker functional $\lambda_{\Psi_{\tilde{w}_0}}$, let $\overline{\theta} = w_0(\theta)$. Theorem~2.15 of \cite{Lo2019} gives that ${\rm I}(\chi,\pi)$ is $\Psi$-generic and establishes a principal value integral for the resulting Whittaker functional
\begin{equation} \label{Whittaker:Ind}
   \lambda_\Psi(\chi,\pi,\tilde{w}_0)f = 
   \int_{N_{\overline{\theta}}} 
   \lambda_{\Psi_{\tilde{w}_0}}(f(\tilde{w}_0^{-1}n)) \overline{\Psi}(n) \, dn.
\end{equation}
Furthermore, from Corollary~2.13 of [\emph{loc.\,cit.}], given a compact open subgroup $K$ of $G$ there exists a compact open subgroup $N_*$ of $N_{\overline{\theta}}$ such that
 \begin{equation} \label{Whittaker:Ind:cpct}
   \lambda_\Psi(\chi,\pi,\tilde{w}_0)f = 
   \int_{N_*}
   \lambda_{\Psi_{\tilde{w}_0}}(f(\tilde{w}_0^{-1}n)) \overline{\Psi}(n) \, dn,
\end{equation}
for every $f \in {\rm V}(\chi,\pi)^K$.

\begin{remark} \label{alphaconvention}
In the Langlands-Shahidi setting we are given a pair $({\bf G},{\bf M}_\theta)$, where ${\bf M}_\theta$ is the Levi of a maximal parabolic subgroup ${\bf P}_\theta$ of $\bf G$. In this situation, we often use $s \in \mathbb{C}$ instead of $\delta_{\!\alpha}^s$ in the notation. We thus simply write $\lambda_\Psi(s,\pi,\tilde{w}_0)$ for $\lambda_\Psi(\delta_{\!\alpha}^s,\pi,\tilde{w}_0)$ and ${\rm I}(s,\pi)$ for ${\rm I}(\delta_{\!\alpha}^s,\pi)$.
\end{remark}

The representation $\tilde{w}_0(\pi)$ of $M_{\overline{\theta}}$ is $\Psi\vert_{U_{\overline{\theta}}}$-generic, with Whittaker functional $\lambda_{\Psi_{\tilde{w}_0}}$. Let $\bar{w}_0 = w_l w_{l,\overline{\theta}}$ be the Weyl group element given by \eqref{w_0} for $\overline{\theta}$, so that $\bar{w}_0 = w_0^{-1}$ and $\bar{w}_0(\overline{\theta}) = \theta$; as in \cite{Sh1990}, the definition below implicitly requires $\Psi$ to be compatible with $\tilde{\bar{w}}_0$ as well. The definition of the local coefficient is made via multiplicity one for Whittaker functionals \cite{Shka1974}, which makes it independent of the test function.

\begin{definition}
Let $\pi$ be a $\Psi_{\tilde{w}_0}$-generic representation of $M = M_\theta$. The Langlands-Shahidi local coefficient $C_\Psi(s,\pi,\tilde{w}_0)$ is defined via the equation
\begin{equation}\label{lslc}
   \lambda_\Psi(s,\pi,\tilde{w}_0) = C_\Psi(s,\pi,\tilde{w}_0) \lambda_\Psi(\tilde{w}_0(\delta_{\alpha})^s,\tilde{w}_0(\pi),\tilde{\bar{w}}_0) {\rm A}(s,\pi,\tilde{w}_0),
\end{equation}
for every $s \in \mathbb{C}$.
\end{definition}

We now indicate how to prove rationality for the local coefficient, stated in \cite{Lo2009}.

\begin{theorem}\label{lc:rationality}
The Langlands-Shahidi local coefficient $C_\Psi(s,\pi,\tilde{w}_0)$ is a rational function of $q_F^{-s}$.
\end{theorem}
\begin{proof}
Rationality for the intertwining operator ${\rm A}(s,\pi,\tilde{w}_0)$ is proved in Th\'eor\`eme 4.1.1 of \cite{Wa2003}. Now, Theorem~2.15 of \cite{Lo2019} establishes that $\lambda_\Psi(s,\pi,\tilde{w}_0)f_s$ is a Laurent polynomial in the variable $q_F^{-s}$, for a test function $f_s$. Hence, from the definition, we conclude that $C_\Psi(s,\pi,\tilde{w}_0)$ is a rational function of $q_F^{-s}$.
\end{proof}

Given a non-trivial character $\psi : F \rightarrow \mathbb{C}^\times$ and a parabolic subgroup ${\bf P} = {\bf M}{\bf N}$, we can always produce non-degenerate characters $\psi$ of $U$ and $\psi_{\tilde{w}_0}$ of $U_M = U \cap M = U_\theta$ which are $\tilde{w}_0$-compatible. Indeed, we obtain a character
\[ \prod_{\alpha \in \Delta} \psi : \prod_{\alpha \in \Delta} N_\alpha / N_{2\alpha} \rightarrow \mathbb{C}^\times, \]
leading to a character of $U$, which we also denote by $\psi$, after composing with the projection map
\[ \psi : U \xrightarrow{\rm proj} \prod_{\alpha \in \Delta} N_\alpha / N_{2\alpha} \rightarrow \mathbb{C}^\times. \]
Then, the non-degenerate character
\[ \psi_{\tilde{w}_0} : U_M \rightarrow \mathbb{C}^\times \]
is defined to satisfy equation~\eqref{nd:char}. Henceforth, for a triple $(F,\pi,\psi) \in \mathfrak{ls}(p,{\bf G},{\bf M})$, we always fix the resulting $\tilde{w}_0$-compatible characters $\psi$ and $\psi_{\tilde{w}_0}$.

\subsection{Rank one cases and compatibility with Tate's thesis}\label{rankone}

Let $\bf G$ be a quasi-split connected reductive group defined over $F$, which is of semisimple $F$-rank one. This means that the maximal $F$-split torus of the derived group ${\bf G}_D$ of $\bf G$ has rank $1$. Let $\widetilde{\bf G}$ be the simply connected cover of $\bf G$ whose derived group $\widetilde{\bf G}_D$ is semisimple of rank one. Hence, $\widetilde{\bf G}_D$ is isomorphic to a group of the following form
\begin{equation}\label{sc:rank1:class}
   {\rm Res}_{F'/F} {\rm SL}_2 \text{ or } {\rm Res}_{F'/F} {\rm SU}_3,
\end{equation}
for  $F'/F$ a separable extension of local fields. This is in such a way that we have a diagram
\begin{equation*}
\begin{tikzcd}[column sep=small]
   1 \arrow[r]	& \widetilde{\bf T} \arrow[r]\arrow[d,swap,"\iota_{\bf T}"]	
   			& \widetilde{\bf G} \arrow[d,swap,"\iota_{\bf G}"] \\
   1 \arrow[r]	& {\bf T} \arrow[r]				& {\bf G}
\end{tikzcd}
\end{equation*}
where $\widetilde{\bf T}$ is a simply connected torus, the vertical arrows are central isogenies and the groups $\bf G$, $\widetilde{\bf G}$ and $\widetilde{\bf G}_D$ have their unipotent radical $\bf U$ in common. We further have a central isogeny on Borel subgroups
\[ \iota_{\bf B} : \widetilde{\bf B} = \widetilde{\bf T}{\bf U} \longrightarrow {\bf B} = {\bf T}{\bf U}. \]
In both cases of \eqref{sc:rank1:class}, $\Delta$ is a singleton $\left\{ \alpha \right\}$ and $U = N_\alpha$. And, we note that the root system of ${\rm Res}_{F'/F}{\rm SU}_3$ has one reduced simple root $\alpha$ and a non reduced root $2\alpha$.

In the case of $\widetilde{\bf G}_D = {\rm Res}_{F'/F} {\rm SL}_2$, the $F$ rational points of the maximal torus $\widetilde{\bf T}_D$ are given by
\begin{equation*}
   \widetilde{T}_D = \left\{ {\rm diag}(t,t^{-1}) \, \vert \, t \in F'^\times \right\},
\end{equation*}
and we pin the isomorphism
\begin{equation*}
   N_\alpha \simeq {\bf G}_a(F') = F',
\end{equation*}
to be the one sending
\begin{equation*}
   u(x) = \left( 
   \begin{array}{cc} 1 & x \\ 0 & 1 \end{array} \right) \in N_\alpha \mapsto x \in F'.
\end{equation*}
We also fix the choice of Weyl group element representative to be
\begin{equation}
   \tilde{w}_\alpha = \left( \begin{array}{rr} 0 & 1 \\ -1 & 0 \end{array} \right).
\end{equation}

In the case of $\widetilde{\bf G}_D = {\rm Res}_{F'/F} {\rm SU}_3$, the semisimple group ${\rm SU}_3$ is given by the standard Hermitian form $h$ for the unitary group in three variables as in \S~4.4.5 of \cite{CoSGA3}. There is here an underlying quadratic extension $E$ over the field $F'$. We fix the $F$ rational points of the maximal torus $\widetilde{\bf T}_D$ to be given in this case by
\begin{equation*}
   \widetilde{T}_D = \left\{ {\rm diag}(t,\bar{t}t^{-1},\bar{t}^{-1}) \, \vert \, t \in E^\times \right\}.
\end{equation*}
Given $x,z \in E$, let
\begin{equation*}
   u(x,z) = \left( \begin{array}{ccc} 1 & \, x & z \\ 0 & 1 & \bar{x} \\ 0 & 0 & 1
 \end{array} \right),
\end{equation*}
then the unipotent radical of ${\rm Res}_{F'/F} {\rm SU}_3(F)$ is given by
\begin{equation}
   N_\alpha = \left\{ u(x,z) \mid x,z \in E \text{ such that } {\rm Tr}_{E/F'}(z) = {\rm N}_{E/F'}(x) \right\}.
\end{equation}
The subgroup
\[ N_{2\alpha} = \{ u(0,z) \mid {\rm Tr}_{E/F'}(z) = 0 \} \]
is identified, via $u(0,z) \mapsto z$, with the kernel $E^0$ of ${\rm Tr}_{E/F'} : E \rightarrow F'$. Furthermore, we pin the isomorphism
\begin{equation}\label{iso:SU:Ualpha}
   N_\alpha / N_{2\alpha} \cong {\bf G}_a(E),
\end{equation}
by requiring that it send on rational points the image of $u(x,z)$ in $N_\alpha / N_{2\alpha}$ to $x \in E$.
For the Weyl group element representative for ${\rm SU}_3(F)$, we fix
\begin{equation}\label{w:SU3}
   \tilde{w}_\alpha = \left( \begin{array}{ccc} 0 & \, 0 & 1 \\ 0 & \!\!\! -1 & 0 \\ 1 & \, 0 & 0 \end{array} \right).
\end{equation}

Given a fixed non-trivial character $\psi: F \rightarrow \mathbb{C}^\times$, and a separable extension $F'/F$, we extend $\psi$ to a character of $F'$ via the trace, i.e., $\psi_{F'} = \psi \circ {\rm Tr}_{F'/F}$; and we similarly get a character $\psi_E$ of $E$. We also have self-dual Haar measures $\mu_{\psi_{F'}}$, $\mu_{\psi_E}$ on $F' = {\bf G}_a(F')$, $E = {\bf G}_a(E) \cong N_\alpha/N_{2\alpha}$, respectively. On $E^0 \cong N_{2\alpha}$ we take the quotient of $\mu_{\psi_E}$ by $\mu_{\psi_{F'}}$ with respect to the exact sequence
\[ 0 \rightarrow E^0 \rightarrow E \xrightarrow{{\rm Tr}} F' \rightarrow 0, \]
and on $N_\alpha$ the Haar measure compatible with $N_{2\alpha}$ and $N_\alpha / N_{2\alpha}$.

In short, let $\bf H$ be either ${\rm SL}_2$ or ${\rm SU}_3$ defined over $F'$. And, let $\widetilde{\bf G} = {\rm Res}_{F'/F}{\bf H}$. Then there is an isomorphism of pinned groups over $F$, and we have an isomorphism
\[ \eta : \widetilde{\bf G}(F) \rightarrow {\bf H}(F'), \]
which gives isomorphisms for the unipotent subgroups. With the normalizations for ${\bf H}$ over $F'$, we then have fixed non-degenerate characters and Haar measures.

\begin{lemma}
Given a triple $(F,\pi,\psi) \in \mathfrak{ls}(p,{\bf G},{\bf T})$, let $(F,\pi',\psi) \in \mathfrak{ls}(p,\widetilde{\bf G},\widetilde{\bf T})$ be the triple obtained via $\iota_{\bf T} : \widetilde{\bf T} \rightarrow {\bf T}$, i.e., $\pi' = \pi \circ \iota_T$. Then
\begin{equation*}
   C_\psi(s,\pi,\tilde{w}_0) = C_\psi(s,\pi',\tilde{w}_0).
\end{equation*}
\end{lemma}
\begin{proof}
On $\widetilde{\bf G}_D$ we have pinned isomorphisms on ${\bf U}$, fixed in the above discussion. We have a commuting diagram
\begin{center}
\begin{tikzpicture}
   \draw (0,0)  node {$1$};
   \draw[->,>=latex] (.25,0) -- (.75,0);
   \draw (1.1,0) node {${T}$};
   \draw[->,>=latex] (1.4,0) -- (1.9,0);
   \draw (2.2,0) node {${B}$};
   \draw (0,1.25)  node {$1$};
   \draw[->,>=latex] (.25,1.25) -- (.75,1.25);
   \draw (1.1,1.25) node {$\widetilde{T}$};
   \draw[->,>=latex] (1.4,1.25) -- (1.9,1.25);
   \draw (2.2,1.25) node {$\widetilde{B}$};
   \draw[<-,>=latex] (1.1,.3) -- (1.1,.9);
   \draw (.8,.6) node {$\iota_{T}$};
   \draw[<-,>=latex] (2.2,.3) -- (2.2,.9);
   \draw (1.9,.6) node {$\iota_{B}$};
   \draw (3.3,0) node {${G}$};
   \draw[<-,>=latex] (3.3,.3) -- (3.3,.9);
   \draw (3,.6) node {$\iota_{G}$};
   \draw (3.3,1.25) node {$\widetilde{G}$};
   \draw[->,>=latex] (2.5,0) -- (3,0);
   \draw[->,>=latex] (2.5,1.25) -- (3,1.25);   
\end{tikzpicture}
\end{center}
which leads us to morphisms compatible with induction and restriction. In particular, we have
\[ {\rm I}(s,\pi) \rightarrow {\rm I}(s,\pi'), \]
compatible with
\[ \pi \mapsto \pi' = \pi \circ \iota_{T}. \]

Now, the computation of the local coefficient depends on the value of $\pi' = \pi \circ \iota_T$ on
\[ {\rm diag}(z^{-1},z) \text{ and } {\rm diag}(\bar{z}^{-1},\bar{z}z^{-1},z), \]
for ${\rm Res}_{F'/F}{\rm SL}_2$ and ${\rm Res}_{F'/F}{\rm SU}_3$, respectively. One simply writes down the defining integrals, and observes that they give the same local coefficients.
\end{proof}

We are thus reduced to the situation where $\bf G$ is a connected quasi-split reductive group of semisimple $F$-rank one defined over $F$, whose derived group ${\bf G}_D$ is exactly one of the two groups found in \eqref{sc:rank1:class}. For the remainder of this section, we assume this to be the case for $\bf G$.

Let $\eta_{E/F'} : F'^\times \rightarrow \mathbb{C}^\times$ be the character of a quadratic extension $E/F'$ obtained via class field theory. More precisely: $\eta_{E/F'}(x) = 1$ if $x \in {\rm N}_{E/F'}(E^\times)$; $\eta_{E/F'}(x) = -1$ if $x \notin {\rm N}_{E/F'}(E^\times)$. We also have Langlands factors $\lambda({F'/F},\psi)$, defined in \cite{LaYale} for any separable extension of non-archimedean local fields $F'/F$. Let $\mathcal{W}_{F'}$ and $\mathcal{W}_F$ be the Weil groups of $F'$ and $F$, respectively. Recall that if $\rho$ is an $n$-dimensional semisimple smooth representation of $\mathcal{W}_{F'}$, then
\begin{equation*}
   \lambda(F'/F,\psi)^n = \dfrac{\varepsilon(s,{\rm Ind}_{\mathcal{W}_{F'}}^{\mathcal{W}_{F}} \rho,\psi)}{\varepsilon(s,\rho,\psi_{F'})}.
\end{equation*}
On the right hand side we have Galois $\varepsilon$-factors, see Chapter~7 of \cite{BuHe} for further properties.

The next result addresses the compatibility of the Langlands-Shahidi local coefficient with the abelian $\gamma$-factors of Tate's thesis \cite{T}. It is essentially Proposition~3.2 of \cite{Lo2016}, which includes the case of ${\rm char}(F) = 2$.

\begin{proposition}\label{rankoneprop}
 Let $(F,\pi,\psi) \in \mathfrak{ls}(p,{\bf G},{\bf T})$, where $\bf G$ is a quasi-split connected reductive group of semisimple $F$-rank one, whose derived group ${\bf G}_D$ is ${\rm Res}_{F'/F} {\rm SL}_2$ or ${\rm Res}_{F'/F} {\rm SU}_3$.
\begin{itemize}
\item[(i)] If ${\bf G}_D = {\rm Res}_{F'/F}{\rm SL}_2$, let $\chi$ be the smooth character of $F'^\times$ given by $\pi({\rm diag}(t,t^{-1})) = \chi(t)$, ${\rm diag}(t,t^{-1}) \in T_D$. Then
\begin{equation*}
   C_\psi(s,\pi,\tilde{w}_0) = \gamma(s,\chi,\psi_{F'}).
\end{equation*}
\item[(ii)] If ${\bf G}_D = {\rm Res}_{F'/F}{\rm SU}_3$, let $\chi$ and $\nu$ be smooth characters of $E^\times$ and $E^1$, respectively, such that
\[ \pi({\rm diag}(t,z,\bar{t}^{-1})) = \chi(t) \nu(z), \ {\rm diag}(t,z,\bar{t}^{-1}) \in T_D. \]
Extend $\nu$ to a character of $E^\times$ via Hilbert's theorem~90. Then
\begin{equation*}
   C_\psi(s,\pi,\tilde{w}_0) = \lambda(E/F',\overline{\psi}_{F'}) \, \gamma_E(s,\chi\nu,\psi_E) \, \gamma(2s, \eta_{E/F'} \chi \vert_{F'^\times},\psi_{F'}).
\end{equation*}
\end{itemize}
\end{proposition}

The computations of \cite{Lo2016} rely mostly on the unipotent group $\bf U$, which is independent of the group $\bf G$. However, there is a difference due to the variation of the maximal torus in the above proposition. For example, if ${\bf G} = {\rm SL}_2$, then all smooth representations $\pi$ of $T$ are of the form $\pi({\rm diag}(t,t^{-1})) = \chi(t)$ for a smooth character $\chi$ of ${\rm GL}_1(F)$; then $C_\psi(s,\pi,\tilde{w}_0) = \gamma(s,\chi,\psi)$. However, in the case of ${\rm GL}_2(F)$ we have $\pi({\rm diag}(t_1,t_2)) = \chi_1(t_1) \chi_2(t_2)$ for smooth characters $\chi_1$ and $\chi_2$ of ${\rm GL}_1(F)$; then $C_\psi(s,\pi,\tilde{w}_0) = \gamma(s,\chi_1 \chi_2^{-1},\psi)$. Note that, up to a direct product with a split torus, the split groups of semisimple rank one are ${\rm PGL}_2$, ${\rm GL}_2$ and ${\rm SL}_2$. The cases ${\rm PGL}_2 \cong {\rm SO}_3$ and ${\rm SL}_2 = {\rm Sp}_2$ are also included in Proposition~3.2 of [\emph{loc.\,cit.}]. The groups ${\rm Res}_{E/F}{\rm GL}_2(F) \cong {\rm GL}_2(E)$, ${\rm U}_2(F)$ and ${\rm U}_3(F)$ arise in connection with the quasi-split unitary groups.

Given a character $\pi$ of $T = {\bf T}(F)$, we have a parameter
\begin{equation}\label{phiparameter}
   \phi: \mathcal{W}_F \rightarrow {}^LT
\end{equation}
corresponding to $\pi$, by the local Langlands correspondence for tori; here and throughout, local class field theory is normalized as in \cite{Ta1979}, so that uniformizers correspond to geometric Frobenius elements. Let $\mathfrak{u}$ denote the Lie algebra of $\bf U$ and let $r$ be the adjoint action of ${}^LT$ on ${}^L\mathfrak{u}$. Then $r$ is irreducible if ${\bf G}_D = {\rm Res}_{F'/F} {\rm SL}_2$ and $r = r_1 \oplus r_2$ if ${\bf G}_D = {\rm Res}_{F'/F} {\rm SU}_3$. As in \cite{KeSh1988}, we normalize Langlands-Shahidi $\gamma$-factors in order to have equality with the corresponding Artin factors.

\begin{definition}
 Let $(F,\pi,\psi) \in \mathfrak{ls}(p,{\bf G},{\bf T})$ be as in Proposition~\ref{rankoneprop}.
\begin{itemize}
\item[(i)] If ${\bf G}_D = {\rm Res}_{F'/F}{\rm SL}_2$, then $r = r_1$ and we let
\begin{equation*}
   \gamma(s,\pi,r,\psi) = \lambda(F'/F,\psi) \gamma(s,\chi,\psi_{F'}).
\end{equation*}
\item[(ii)] If ${\bf G}_D = {\rm Res}_{F'/F}{\rm SU}_3$, let 
\begin{equation*}
   \gamma(s,\pi,r_1,\psi) = \lambda(E/F,\psi) \gamma_E(s,\chi\nu,\psi_E)
\end{equation*}
and
\begin{equation*}
   \gamma(s,\pi,r_2,\psi) = \lambda(F'/F,\psi) \, \gamma(s, \eta_{E/F'} \chi \vert_{F'^\times},\psi_{F'}).
\end{equation*}
\end{itemize}
\end{definition}

In this way, with $\phi$ as in \eqref{phiparameter}, we have for each $i$:
\begin{equation*}
   \gamma(s,\pi,r_i,\psi) = \gamma(s,r_i \circ \phi,\psi).
\end{equation*}
The $\gamma$-factors on the right hand side are those defined by Artin, Deligne and Langlands \cite{Ta1979}. Our choices of Weyl group element representatives and of $\lambda$-factors agree with those of \cite[\S\S~2.6--2.7]{AGIKMS}, where the normalizations of \cite{Sh1990} and \cite{KeSh1988} are reconciled, and Proposition~\ref{rankoneprop} agrees with the rank one computations there.

\section{Normalization of the local coefficient}\label{Weyl}

In this section, we address the effect of varying the non-degenerate character on the Langlands-Shahidi local coefficient, in addition to changing the system of Weyl group element representatives and Haar measures, in Proposition~\ref{vary}. We begin with some preliminaries, including Langlands' lemma, for which we refer to Shahidi's algorithmic proof, Lemma~2.1.2 of \cite{Sh1981}. We then recall the multiplicativity property of the local coefficient and use it to connect between unramified principal series and the rank one cases of Proposition~\ref{rankoneprop}.

\subsection{Weyl group element representatives and Haar measures}\label{fixedW}
Recall that two subsets $\theta$ and $\theta'$ of $\Delta$ are associate if $W(\theta,\theta') = \left\{ w \in W \vert w(\theta) = \theta' \right\}$ is non-empty. Given $w \in W(\theta,\theta')$, define
\begin{equation}\label{N:w:bar}
   N_w = U \cap w N_\theta^- w^{-1}, \quad \overline{N}_w = w^{-1} N_w w.
\end{equation}
Their corresponding Lie algebras are denoted by $\mathfrak{n}_w$ and $\overline{\mathfrak{n}}_w$. The following is known as Langlands' lemma.

\begin{lemma}\label{langlandslemma} Let $\theta, \theta' \subset \Delta$ be associate and let $w \in W(\theta,\theta')$. Then, there exists a family of subsets $\theta_1, \ldots, \theta_d \subset \Delta$ such that:
\begin{itemize}
   \item[(i)] We begin with $\theta_1 = \theta$ and end with $\theta_d = \theta'$.
   \item[(ii)] For each $j$, $1 \leq j \leq d-1$, there exists a root $\alpha_j \in \Delta - \theta_j$ such that $\theta_{j+1}$ is the conjugate of $\theta_j$ in $\Omega_j = \left\{ \alpha_j \right\} \cup \theta_j$.
   \item[(iii)] Set $w_j = w_{l,\Omega_j} w_{l,\theta_j}$ in $W(\theta_j,\theta_{j+1})$ for $1 \leq j \leq d-1$, then $w = w_{d-1} \cdots w_1$.
   \item[(iv)] If one sets $\dot{w}_1 = w$ and $\dot{w}_{j+1} = \dot{w}_j w_j^{-1}$ for $1 \leq j \leq d-1$, then $\dot{w}_d = 1$ and
   \begin{equation*}
      \overline{\mathfrak{n}}_{\dot{w}_j} = \overline{\mathfrak{n}}_{w_j} \oplus {\rm Ad}(w_j^{-1}) \overline{\mathfrak{n}}_{\dot{w}_{j+1}}.
   \end{equation*}
\end{itemize}
\end{lemma}

For each $\alpha \in \Delta$ there corresponds a group ${\bf G}_\alpha \subset {\bf G}$ with a simply connected cover $\widetilde{\bf G}_\alpha$, whose derived group is semisimple of rank one. The groups ${\bf G}_\alpha$ and $\widetilde{\bf G}_\alpha$ share the same unipotent radical ${\bf N}_\alpha$, and we have a simply connected torus $\widetilde{\bf T}_\alpha$. A Weyl group element representative $\tilde{w}_\alpha$ is chosen for each $w_\alpha$ and Haar measures on the unipotent radicals $N_\alpha$ are normalized as indicated in \S~\ref{rankone}. We take the corresponding measure on $N_{-\alpha}$ inside $G_\alpha$. Fix
\begin{equation}\label{locWsystem}
   \mathfrak{W} = \left\{ \tilde{w}_\alpha, d\mu_\alpha \right\}_{\alpha \in \Delta}
\end{equation}
to be this system of Weyl group element representatives in the normalizer ${\rm N}(T_s)$ together with fixed Haar measures on each $N_\alpha$.

We note that the set of roots $\Sigma$ admits a subset of positive roots $\Sigma^+$, according to the fixed choice of simple roots and Borel subgroup. We further have the set of reduced roots $\Sigma_{\rm red}$, and the positive reduced roots $\Sigma_{\rm red}^+$.

We can apply Lemma~\ref{langlandslemma} to the Weyl group element $w_0$ of equation~\eqref{w_0} by taking the Borel subgroup $\bf B$ of $\bf G$, i.e., using $\theta = \emptyset$. In this way, we obtain a decomposition
\begin{equation}\label{w_0decomp}
   w_0 = \prod_{i=1}^d w_{0,\beta_i},
\end{equation}
where the $\beta_i$ range through the set $\Sigma_{\rm red}^+$.  For each such $w_{0,\beta_i}$, there corresponds a simple reflection $w_\alpha$ for some $\alpha \in \Delta$. We observe that the decomposition of \eqref{w_0decomp} is not unique. However, it is a result of Tits that the choice of Weyl group element representatives $\tilde{w}_\alpha$, $\alpha \in \Delta$, determines a unique $\tilde{w}_0$ \cite{Ti1966}.

From Langlands' lemma and equation~\eqref{w_0decomp} we further obtain a decomposition of $N$ in terms of $N_{\beta_i}'$, $\beta_i \in \Sigma_{\rm red}^+$, where each $N_{\beta_i}'$ corresponds to the unipotent group $N_\alpha$ of ${\bf G}_\alpha$, for some $\alpha \in \Delta$. In this way, the measure on $N$ is fixed by $\mathfrak{W}$ and we denote it by $dn$. This also fixes measures on the groups $N_{w_0}$ and $\overline{N}_{w_0}$ of \eqref{N:w:bar}, in addition to any $N_\theta$, $\theta \subset \Delta$.

Before continuing, we introduce additional notation. Let ${\bf N} \subset {\bf U}$ be a subgroup generated by root subgroups, with Lie algebra $\mathfrak{n}$, and let $y \in S$, where ${\bf S}$ is a torus normalizing ${\bf N}$. Denote the module of the automorphism $n \mapsto y n y^{-1}$ of $N$ by $c(y;N)$, then
\begin{equation}\label{module}
   c(y;N) = \left| \det({\rm Ad}(y) \vert_{\mathfrak{n}}) \right|_F.
\end{equation}
Notice that $y \mapsto c(y;N)$ is multiplicative. And, in particular, $c(y;N_\theta) = \delta_{P_\theta}(y)$ for $y \in T$.

\begin{proposition}\label{vary} Let $(F,\pi,\psi) \in \mathfrak{ls}(p,{\bf G},{\bf M})$. Let $\mathfrak{W}' = \left\{ \tilde{w}_\alpha', d\mu_\alpha' \right\}_{\alpha \in \Delta}$ be an arbitrary system of Weyl group element representatives and Haar measures. Let $\phi: U \rightarrow \mathbb{C}^\times$ and $\phi_{\tilde{w}_0'} : U_M \rightarrow \mathbb{C}^\times$ be non-degenerate $\tilde{w}_0'$-compatible characters and assume that $\pi$ is $\phi_{\tilde{w}_0'}$-generic. There is a connected quasi-split reductive group $\widetilde{\bf G}$ defined over $F$, sharing the same derived group as $\bf G$, whose center is connected with trivial $H^1$, and with maximal torus $\widetilde{\bf T} = {\bf Z}_{\widetilde{\bf G}} {\bf T}$. Then, there exists an element $x \in \widetilde{T}$ such that the representation $\pi_x$, given by
\begin{equation*}
   \pi_x(g) = \pi(x^{-1}gx)
\end{equation*}
is $\psi_{\tilde{w}_0}$-generic. And, there exists a monomial $a_x(s,\phi,\mathfrak{W}')$ in $q_F^{-s}$ such that
\begin{equation*}
   C_\phi(s,\pi,\tilde{w}_0') = a_x(s,\phi,\mathfrak{W}') C_\psi(s,\pi_x,\tilde{w}_0).
\end{equation*}
Let $(F,\pi_i,\psi) \in \mathfrak{ls}(p,{\bf G},{\bf M})$, $i = 1$, $2$. If $\pi_1 \cong \pi_2$ and both are $\psi_{\tilde{w}_0}$-generic, then
\begin{equation*}
   C_\psi(s,\pi_1,\tilde{w}_0) = C_\psi(s,\pi_2,\tilde{w}_0).
\end{equation*}
\end{proposition}

\begin{proof}
The existence of a connected quasi-split reductive group $\widetilde{\bf G}$, whose center is connected with trivial $H^1$, sharing the same derived group as $\bf G$, is due thanks to Proposition~5.4 of \cite{Sh2002}. Its maximal torus is given by $\widetilde{\bf T} = {\bf Z}_{\widetilde{\bf G}}{\bf T}$. Let $(F,\pi,\psi) \in \mathfrak{ls}(p)$, where $(\pi,V)$ is $\phi_{\tilde{w}_0'}$-generic. Because the center of $\widetilde{\bf G}$ has trivial $H^1$, the character $\phi$ lies in the same orbit as $\psi$. Thus, there exists an $x \in \widetilde{T}$ such that $\phi(n) = \psi(x n x^{-1})$, $n \in U$; we show below that $\pi_x$ is then $\psi_{\tilde{w}_0}$-generic.

The system $\mathfrak{W}'$ fixes measures on $N_{\overline{\theta}}$ and $N_\theta$, which we denote by $dn'$. Uniqueness of Haar measures gives constants $b, b' \in \mathbb{C}$ such that $dn = b \, dn'$ on $N_{\overline{\theta}}$ and $dn = b' \, dn'$ on $N_\theta$.

To work with the local coefficient, take $\varphi \in C_c^\infty(P_\theta w_0 B, V)$ and let $f = f_s = \mathcal{P}_s \varphi$ be as in Proposition~2.5 of \cite{Lo2019}, with $\nu = \delta_\alpha^s$. Now, using the system $\mathfrak{W}'$ in the right hand side of the definition~\eqref{lslc} for $C_\phi(s,\pi,\tilde{w}_0')$, we have
\begin{align*}
   \lambda_\phi &(\tilde{w}_0'(\delta_\alpha)^s,\tilde{w}_0'(\pi),\tilde{\bar{w}}_0') {\rm A}(s,\pi,\tilde{w}_0') f \\
      & = \int_{N_\theta} \int_{N_{w_0}} \lambda_{\phi_{\tilde{w}_0'}} \left( f(\tilde{w}_0'^{-1} n_1 \tilde{\bar{w}}_0'^{-1} n_2) \right) \overline{\phi}(n_2) \, dn_1' \, dn_2'.
\end{align*}
Here $x \in \widetilde{T}$ is as above, so that $\phi(n) = \psi(x n x^{-1})$, $n \in U$, and we let
\begin{equation}\label{eq:e:element}
   e = x \tilde{w}_0'^{-1} x^{-1} \tilde{w}_0'.
\end{equation}
For $y \in \widetilde{T}$, write $y^w = \tilde{w}_0^{-1} y \tilde{w}_0$. Since two Weyl group element representatives of $w_0$ differ by an element of $T$ and $\widetilde{T}$ is commutative, $y^w$ does not depend on the choice of representative. Similarly, since $\bar{w}_0 = w_0^{-1}$, conjugation by $\tilde{\bar{w}}_0$ or by $\tilde{\bar{w}}_0'$ acts on $\widetilde{T}$ as $y \mapsto y^w$, and conjugation by their inverses as $y \mapsto \tilde{w}_0 y \tilde{w}_0^{-1}$. In particular, $e = x (x^w)^{-1}$. The representative $\tilde{w}_0$ lies in $G$, which is normal in $\widetilde{G}$, hence $x$ and $x^w$ have the same image in $\widetilde{G}/G$. Therefore $e \in G \cap \widetilde{T} = T$.

Let $f_x(g) = f(x^{-1}gx)$, so that $f_x \in {\rm V}(s,\pi_x)$, and let $\rho = \delta_\alpha^s \delta_P^{1/2}$, so that $f_x(mg) = \rho(m) \pi_x(m) f_x(g)$, $m \in M$. We write $\tilde{w}_0' = \tilde{w}_0 t$ for some $t \in T$, and let
\begin{equation}\label{eq:bar:t:zero}
   \bar{t}_0 = \bar{t}^w, \ \text{with} \ \bar{t} \in T \ \text{such that} \ \tilde{\bar{w}}_0' = \tilde{\bar{w}}_0 \bar{t}.
\end{equation}
Since $N_{\overline{\theta}} = \tilde{w}_0 N_\theta^- \tilde{w}_0^{-1}$, transport of structure by $\tilde{w}_0$ and \eqref{module} give
\begin{equation*}
   c(y;N_{\overline{\theta}}) = c(y^w;N_\theta^-) = c(y^w;N_\theta)^{-1}, \ y \in \widetilde{T}.
\end{equation*}

\emph{The Whittaker functionals.} Let $t_1 = t e^{-1} \in T$. From \eqref{eq:e:element} we have $x \tilde{w}_0' x^{-1} = \tilde{w}_0' e^{-1}$. Hence, for $u \in U_M$ and $v \in V$,
\begin{align*}
   \lambda_{\phi_{\tilde{w}_0'}}(\pi_x(u)v) &= \phi(\tilde{w}_0' x^{-1} u x \tilde{w}_0'^{-1}) \lambda_{\phi_{\tilde{w}_0'}}(v)
   = \psi(\tilde{w}_0' e^{-1} u e \tilde{w}_0'^{-1}) \lambda_{\phi_{\tilde{w}_0'}}(v) \\
   &= \psi_{\tilde{w}_0}(t_1 u t_1^{-1}) \lambda_{\phi_{\tilde{w}_0'}}(v).
\end{align*}
Consequently, $\lambda_{\psi_{\tilde{w}_0}} = \lambda_{\phi_{\tilde{w}_0'}} \circ \pi_x(t_1)^{-1}$ is a $\psi_{\tilde{w}_0}$-Whittaker functional for $\pi_x$. In particular, $\pi_x$ is $\psi_{\tilde{w}_0}$-generic and, for every $g \in G$,
\begin{equation}\label{lambda:shift}
   \lambda_{\phi_{\tilde{w}_0'}}(f_x(g)) = \lambda_{\psi_{\tilde{w}_0}}(\pi_x(t_1) f_x(g)) = \rho(t_1)^{-1} \lambda_{\psi_{\tilde{w}_0}}(f_x(t_1 g)).
\end{equation}

\emph{A central element.} The right hand side of \eqref{lslc} for $C_\phi(s,\pi,\tilde{w}_0')$ involves the Whittaker functional $\lambda_{\phi_{\tilde{w}_0'}}$ of $\tilde{w}_0'(\pi)$ in \eqref{Whittaker:Ind} for $P_{\overline{\theta}}$ and $\tilde{\bar{w}}_0'$. This requires $\phi(\tilde{\bar{w}}_0' u \tilde{\bar{w}}_0'^{-1}) = \phi(u)$, $u \in U_{\overline{\theta}}$, which is implicit in the definition of the local coefficient, see \S~3 of \cite{Sh1990}; for $\psi$ and $\tilde{\bar{w}}_0$, it holds by the property of Tits' representatives recalled before Lemma~\ref{pinning}. Equivalently, $\phi(\tilde{\bar{w}}_0'^{-1} v \tilde{\bar{w}}_0') = \phi(v)$, $v \in U_M$. Now, $x \tilde{\bar{w}}_0'^{-1} = \tilde{\bar{w}}_0'^{-1} x^w$ and $\tilde{\bar{w}}_0'^{-1} = \tilde{\bar{w}}_0^{-1} \bar{t}_0^{-1}$, so that
\begin{equation*}
   \phi(\tilde{\bar{w}}_0'^{-1} v \tilde{\bar{w}}_0') = \psi(\tilde{\bar{w}}_0^{-1} \, \bar{t}_0^{-1} x^w v (\bar{t}_0^{-1} x^w)^{-1} \, \tilde{\bar{w}}_0) = \psi(\bar{t}_0^{-1} x^w v (\bar{t}_0^{-1} x^w)^{-1}), \ v \in U_M.
\end{equation*}
Comparing with $\phi(v) = \psi(x v x^{-1})$ and replacing $v$ by $x^{-1} v x$, we obtain that
\begin{equation*}
   z_1 = e \, \bar{t}_0 \in T
\end{equation*}
satisfies $\psi(z_1^{-1} v z_1) = \psi(v)$, $v \in U_M$. Since $\psi$ is non-trivial and the trace forms are non-degenerate, $z_1$ acts trivially on the pinned simple root groups of $\bf M$. Hence every root of $\bf M$ is trivial on $z_1$, and $z_1 \in Z_M$.

\emph{The left hand side of \eqref{lslc}.} We have $dn' = b^{-1} dn$ on $N_{\overline{\theta}}$, and the change of variables $n \mapsto x^{-1} n x$ contributes $c(x;N_{\overline{\theta}})^{-1}$. Since $\phi(x^{-1} n x) = \psi(n)$ and $f(\tilde{w}_0'^{-1} x^{-1} n x) = f_x(e \tilde{w}_0'^{-1} n)$, we obtain from \eqref{lambda:shift} and $t_1 e \tilde{w}_0'^{-1} = t \tilde{w}_0'^{-1} = \tilde{w}_0^{-1}$ that
\begin{align*}
   \lambda_\phi(s,\pi,\tilde{w}_0') f & = \dfrac{1}{b \, c(x;N_{\overline{\theta}})} \int_{N_{\overline{\theta}}}
   \lambda_{\phi_{\tilde{w}_0'}} \left( f_x(e \tilde{w}_0'^{-1}n) \right) \overline{\psi} (n) \, dn \\
   & = \dfrac{1}{b \, c(x;N_{\overline{\theta}}) \, \rho(t_1)} \int_{N_{\overline{\theta}}}
   \lambda_{\psi_{\tilde{w}_0}} \left( f_x(\tilde{w}_0^{-1}n) \right) \overline{\psi} (n) \, dn
   = \dfrac{\lambda_\psi(s,\pi_x,\tilde{w}_0) f_x}{b \, c(x;N_{\overline{\theta}}) \, \rho(t_1)}.
\end{align*}

\emph{The right hand side of \eqref{lslc}.} In the double integral above, we have $dn_1' = b^{-1} dn_1$ and $dn_2' = b'^{-1} dn_2$. We make the changes of variables $n_i \mapsto x^{-1} n_i x$, $i = 1$, $2$, which contribute $c(x;N_{\overline{\theta}})^{-1} c(x;N_\theta)^{-1}$. The integrand becomes
\begin{equation*}
   \lambda_{\phi_{\tilde{w}_0'}} \left( f_x(e \tilde{w}_0'^{-1} n_1 \bar{e} \tilde{\bar{w}}_0'^{-1} n_2) \right) \overline{\psi}(n_2), \quad \bar{e} = x \tilde{\bar{w}}_0'^{-1} x^{-1} \tilde{\bar{w}}_0'.
\end{equation*}
We have $\bar{e} = x \tilde{w}_0 x^{-1} \tilde{w}_0^{-1} \in T$ and $\bar{e}^w = x^w x^{-1}$; hence $\tilde{w}_0'^{-1} \bar{e} \tilde{w}_0' = \bar{e}^w = e^{-1}$, i.e., $e \tilde{w}_0'^{-1} \bar{e} = \tilde{w}_0'^{-1}$. We move $\bar{e}$ to the left via $n_1 \mapsto \bar{e} n_1 \bar{e}^{-1}$, which contributes $c(\bar{e};N_{\overline{\theta}})$, and the integrand becomes $\lambda_{\phi_{\tilde{w}_0'}}(f_x(\tilde{w}_0'^{-1} n_1 \tilde{\bar{w}}_0'^{-1} n_2)) \overline{\psi}(n_2)$. Next,
\begin{equation*}
   \tilde{w}_0'^{-1} n_1 \tilde{\bar{w}}_0'^{-1} = t^{-1} \bar{t}_0^{-1} \tilde{w}_0^{-1} (\bar{t} n_1 \bar{t}^{-1}) \tilde{\bar{w}}_0^{-1},
\end{equation*}
and the change of variables $n_1 \mapsto \bar{t}^{-1} n_1 \bar{t}$ contributes $c(\bar{t};N_{\overline{\theta}})^{-1}$. By \eqref{lambda:shift}, with $t_1 t^{-1} \bar{t}_0^{-1} = e^{-1} \bar{t}_0^{-1} = z_1^{-1}$, and since $f_x(z_1^{-1} g) = \rho(z_1)^{-1} \omega_\pi(z_1)^{-1} f_x(g)$ for $z_1 \in Z_M$, the integrand becomes
\begin{equation*}
   \dfrac{1}{\rho(t_1) \, \rho(z_1) \, \omega_\pi(z_1)} \, \lambda_{\psi_{\tilde{w}_0}} \left( f_x(\tilde{w}_0^{-1} n_1 \tilde{\bar{w}}_0^{-1} n_2) \right) \overline{\psi}(n_2).
\end{equation*}
Therefore, the right hand side of \eqref{lslc} for $C_\phi(s,\pi,\tilde{w}_0')$ is equal to
\begin{equation*}
   \dfrac{c(\bar{e};N_{\overline{\theta}}) \cdot \lambda_\psi (\tilde{w}_0(\delta_\alpha)^s, \tilde{w}_0(\pi_x),\tilde{\bar{w}}_0) {\rm A}(s,\pi_x,\tilde{w}_0) f_x}{b b' c(x;N_{\overline{\theta}}) c(x;N_\theta) c(\bar{t};N_{\overline{\theta}}) \, \rho(t_1) \, \rho(z_1) \, \omega_\pi(z_1)}.
\end{equation*}

\emph{The monomial.} By observing that $\bar{e}^w = x^w x^{-1}$, $e = x (x^w)^{-1}$, and using \eqref{module}, we infer that
\begin{equation*}
   c(\bar{e};N_{\overline{\theta}}) = c(x^w x^{-1};N_\theta)^{-1} = c(x;N_\theta) c(x;N_{\overline{\theta}}) = c(e;N_\theta) = \delta_P(e),
\end{equation*}
and also $c(\bar{t};N_{\overline{\theta}}) = c(\bar{t}_0;N_\theta)^{-1} = \delta_P(\bar{t}_0)^{-1}$. Comparing both sides of \eqref{lslc}, the factors $b$, $\rho(t_1)$ and $c(x;N_\theta)$ cancel, and
\begin{equation*}
   C_\phi(s,\pi,\tilde{w}_0') = \dfrac{b' \, \omega_\pi(z_1) \, \rho(z_1)}{c(x;N_{\overline{\theta}}) \, \delta_P(\bar{t}_0)} \, C_\psi(s,\pi_x,\tilde{w}_0).
\end{equation*}
Finally, $\rho(z_1) = \delta_\alpha(z_1)^s \delta_P(e)^{1/2} \delta_P(\bar{t}_0)^{1/2}$, with $\delta_P(e) = c(x;N_\theta) c(x;N_{\overline{\theta}})$. In this way, we arrive at the monomial
\begin{equation}\label{themonomial}
   a_x(s,\phi,\mathfrak{W}') = b' \, \omega_\pi(z_1) \, \delta_\alpha(z_1)^s \left( \dfrac{c(x;N_\theta)}{c(x;N_{\overline{\theta}})} \right)^{1/2} \delta_P(\bar{t}_0)^{-1/2}, \quad z_1 = e \, \bar{t}_0,
\end{equation}
where $e$ is defined by \eqref{eq:e:element} and $\bar{t}_0$ by \eqref{eq:bar:t:zero}. To conclude, we notice that if $(F,\pi_i,\psi) \in \mathfrak{ls}(p)$, $i = 1$, $2$, have $\pi_1 \cong \pi_2$ and are both $\psi_{\tilde{w}_0}$-generic, then Proposition~3.1 of \cite{Sh1981} gives
\begin{equation*}
   C_\psi(s,\pi_1,\tilde{w}_0) = C_\psi(s,\pi_2,\tilde{w}_0).
\end{equation*}
\end{proof}

The monomial of Proposition~\ref{vary} is explicitly given by equation~\eqref{themonomial}. In practice, the most common usage is when $\mathfrak{W}'$ has the same characters and Haar measures as $\mathfrak{W}$, while the representatives differ only by signs. Then the term $\delta_\alpha(z_1)^s$ appearing in \eqref{themonomial} is $1$, so in that case $a_x(s,\phi,\mathfrak{W}')$ is a constant. In general, however, $a_x(s,\phi,\mathfrak{W}')$ is a genuine monomial, as reflected in Property~(v) of Theorem~\ref{mainthm}.

\begin{proposition}\label{x:independence}
Let $(F,\pi,\psi) \in \mathfrak{ls}(p,{\bf G},{\bf M})$, and let $\widetilde{\bf G}$ and $\widetilde{\bf T} = {\bf Z}_{\widetilde{\bf G}}{\bf T}$ be as in Proposition~\ref{vary}. If $x, x' \in \widetilde{T}$ are such that $\pi_x$ and $\pi_{x'}$ are both $\psi_{\tilde{w}_0}$-generic, then $\pi_{x'} \cong \pi_x$. In particular,
\[ C_\psi(s,\pi_{x'},\tilde{w}_0) = C_\psi(s,\pi_x,\tilde{w}_0), \]
and, in the notation of Proposition~\ref{vary}, $a_{x'}(s,\phi,\mathfrak{W}') = a_x(s,\phi,\mathfrak{W}')$.
\end{proposition}
\begin{proof}
Let $c = x^{-1}x'$, so that $\pi_{x'} = \pi_{xc} = \pi_x \circ {\rm Ad}(c^{-1})$. Let $\widetilde{\bf M} = {\bf Z}_{\widetilde{\bf G}}{\bf M}$, the Levi subgroup of $\widetilde{\bf G}$ corresponding to $\bf M$. The inclusion ${\bf M} \subseteq \widetilde{\bf M}$ induces the identity on derived groups. Hence, by Corollary~3.5 and \S~2.1.3 of \cite{HeVi}, there is an irreducible smooth representation $\Pi$ of $\widetilde{M}$ whose restriction to $M$ is a finite direct sum of conjugates $\pi_x \circ {\rm Ad}(y^{-1})$, $y \in \widetilde{M}$, among which $\pi_x$ occurs; in characteristic zero, see also Appendix~A.5 of \cite{AGIKMS}. Since $c \in \widetilde{T} \subset \widetilde{M}$, $\pi_{x'}$ also occurs in $\Pi\vert_M$. Since $U_M \subset M$ and $\Pi$ is irreducible, multiplicity one for the quasi-split group $\widetilde{M}$ \cite{Shka1974} gives
\[ \dim {\rm Hom}_{U_M}(\Pi,\psi_{\tilde{w}_0}) \leq 1. \]
The representations $\pi_x$ and $\pi_{x'}$ are both $\psi_{\tilde{w}_0}$-generic constituents of $\Pi\vert_M$, hence $\pi_{x'} \cong \pi_x$.

The equality of local coefficients now follows from the last statement of Proposition~\ref{vary}, and the equality $a_{x'}(s,\phi,\mathfrak{W}') = a_x(s,\phi,\mathfrak{W}')$ from
\[ C_\phi(s,\pi,\tilde{w}_0') = a_x(s,\phi,\mathfrak{W}') C_\psi(s,\pi_x,\tilde{w}_0). \]
\end{proof}

\subsection{$L$-groups and Langlands-Shahidi representations}\label{Lgrp:LSrep} Given our connected reductive quasi-split group ${\bf G}$ over a non-archimedean local field or a global function field, it is also a group over its separable algebraic closure. Let $\mathcal{W}$ be the corresponding Weil group. The pinning of the roots determines a based root datum $\Psi_0 = (X^*,\Delta, X_*,\Delta^\vee)$. The dual root datum $\Psi_0^\vee = (X_*,\Delta^\vee,X^*,\Delta)$ determines the Chevalley group ${}^LG^\circ$ over $\mathbb{C}$. Then the $L$-group of ${\bf G}$ is the semidirect product
\begin{equation*}
   {}^LG = {}^LG^\circ \rtimes \mathcal{W},
\end{equation*}
with details given in \cite{Bo1979}. The based root datum $\Psi_0^\vee$ fixes a Borel subgroup ${}^LB$, and we have all standard parabolic subgroups of the form ${}^LP = {}^LP^\circ \rtimes \mathcal{W}$. The Levi subgroup of ${}^LP$ is of the form ${}^LM = {}^LM^\circ \rtimes \mathcal{W}$, while the unipotent radical is given by ${}^LN = {}^LN^\circ$. 

Let $R:{}^LM \rightarrow {\rm End}({}^L\mathfrak{n})$ be the adjoint representation of ${}^LM$ on the Lie algebra ${}^L\mathfrak{n}$ of ${}^LN$. It decomposes into $m = m_R$ irreducible components
\begin{equation*}
   R = \oplus_{i=1}^{m} r_i.
\end{equation*}
The $r_i$'s are ordered according to nilpotency class, something that we now proceed to make more precise.

Let ${\bf P} = {\bf M}{\bf N}$ be the maximal parabolic associated to the simple root $\alpha \in \Delta$, i.e., ${\bf P} = {\bf P}_\theta$ for $\theta =\Delta - \left\{ \alpha \right\}$. Furthermore, observe that any non-trivial $\delta \in {\rm Hom}_F({\bf M}/{\bf Z}_{\bf G},{\rm GL}_1)$, gives a uniquely determined
\begin{equation}\label{alphatilde}
   \widetilde{\alpha} = \left\langle \delta, \alpha^\vee \right\rangle^{-1} \delta.
\end{equation}
Now, consider the eigenspaces of ${}^LM^\circ$ given by
\begin{equation*}
   {}^L\mathfrak{n}_i = \left\{ X_{\beta^\vee} \in {}^L\mathfrak{n}
   \ \vert \, \left\langle \widetilde{\alpha}, \beta^\vee \right\rangle = i \right\}, 
   \ 1 \leq i \leq m_R.
\end{equation*}
Then each $r_i$ is a representation with complex vector space ${}^L\mathfrak{n}_i$. We say that an irreducible representation $r$ of ${}^LM$ is a \emph{Langlands-Shahidi representation} if $r = r_i$ for some $i$, $1 \leq i \leq m_R$; in this case we write $i_r = i$, so that $i_r = \left\langle \widetilde{\alpha}, \beta^\vee \right\rangle$ for every $X_{\beta^\vee} \in {}^L\mathfrak{n}_r$, and ${}^L\mathfrak{n}_r = {}^L\mathfrak{n}_{i_r}$ is the complex vector space associated to $r$.

\subsection{Induction and multiplicativity}\label{mult:ind} Shahidi's algorithm allows us to obtain a block based version of Langlands' lemma from the Corollary to Lemma~2.1.2 of \cite{Sh1981}. It leads us to a simple proof of an important inductive argument, Lemma~\ref{induction}. We summarize the necessary results in this section.

Let ${\bf P} = {\bf M}{\bf N}$ be maximal corresponding to $\theta \subset \Delta$. Consider a subset $\theta_0 \subset \theta$ and its corresponding parabolic subgroup ${\bf P}_{\theta_0}$, with Levi ${\bf M}_{\theta_0}$ and unipotent radical ${\bf N}_{\theta_0}$.

Let $\Sigma(\theta_0)$ be the roots of $({\bf P}_{\theta_0},{\bf A}_{\theta_0})$. More precisely, let $\Sigma^+({\bf A}_{\theta_0},{\bf M}_{\theta_0})$ be the positive roots with respect to the maximal split torus ${\bf A}_{\theta_0}$ in the center of ${\bf M}_{\theta_0}$. We say that $\alpha,\beta \in \Sigma^+ - \Sigma^+({\bf A}_{\theta_0},{\bf M}_{\theta_0})$ are ${\bf A}_{\theta_0}$-equivalent if $ \beta\vert_{A_{\theta_0}} = \alpha\vert_{A_{\theta_0}}$. Then
\begin{equation*}
   \Sigma^+(\theta_0) = \left( \Sigma^+ - \Sigma^+({\bf A}_{\theta_0},{\bf M}_{\theta_0}) \right) / \sim.
\end{equation*}
Let $\Sigma_{\rm red}^+(\theta_0)$ be the block reduced roots in $\Sigma^+(\theta_0)$. With the notation of Langlands' lemma, set
\begin{equation*}
   \Sigma(\theta_0,w) = \left\{ [\beta] \in \Sigma^+(\theta_0) \vert  w(\beta) \in \Sigma^- \right\},
\end{equation*}
where $\Sigma^- = -\Sigma^+$, and the corresponding set of reduced roots
\begin{equation*}
   \Sigma_{\rm red}(\theta_0,w) = \left\{ [\beta] \in \Sigma_{\rm red}^+(\theta_0) \vert  w(\beta) \in \Sigma^- \right\}.
\end{equation*}
We have that
\begin{equation}\label{betai}
   [\beta_j] = w_1^{-1} \cdots w_{j-1}^{-1}([\alpha_j]), \ 1 \leq j \leq d-1,
\end{equation}
are all distinct in $\Sigma_{\rm red}(\theta_0,w)$ and all $[\beta] \in \Sigma_{\rm red}(\theta_0,w)$ are obtained in this way. Furthermore, from Lemma~\ref{langlandslemma}(iv), we have
\begin{equation}\label{Nsemidirect}
   \overline{N}_{\dot{w}_j} = {\rm Ad}(w_j^{-1}) \overline{N}_{\dot{w}_{j+1}} \rtimes \overline{N}_{w_j}.
\end{equation}

We return to the Weyl group element $w_0$ of equation~\eqref{w_0}, and apply Lemma~\ref{langlandslemma} to $w_0 \in W(\theta_0,\theta_0')$, $\theta_0' = w_0(\theta_0)$. We obtain a block based decomposition
\begin{equation}\label{blockdecomp}
   w_0 = w_{d-1} \cdots w_1.
\end{equation}
In addition, the unipotent group ${\bf N}_{w_0}$ decomposes into a product via successive applications of equation~\eqref{Nsemidirect}. Namely
\begin{equation}\label{Nblockdecomp}
   {\bf N}_{w_0} = \prod_{j} {\bf N}_{0,j},
\end{equation}
where each ${\bf N}_{0,j}$ is the block unipotent subgroup of $\bf G$ corresponding to the block $[\beta_j]$ of \eqref{betai} and its positive multiples in $\Sigma(\theta_0,w_0)$. Explicitly
\begin{equation}\label{Nblock}
   \overline{\bf N}_{0,j} = {\rm Ad}(w_1^{-1} \cdots w_{j-1}^{-1}) \overline{\bf N}_{w_j},
\end{equation}
which gives a block unipotent ${\bf N}_{0,j}$ of $\bf G$. For each $j$ we have an embedding
\begin{equation*}
   {\bf G}_{j} \hookrightarrow {\bf G}
\end{equation*}
of connected quasi-split reductive groups. Each constituent of this decomposition of $N_{w_0}$ is isomorphic to a block unipotent subgroup $N_{w_j}$ of ${\bf G}_{j}$. The reductive group ${\bf M}_{j}$ has root system $\theta_{j}$ and is a maximal Levi subgroup of the reductive group ${\bf G}_{j}$ with root system $\Omega_{j}$. We have ${\bf P}_{j} = {\bf M}_{j} {\bf N}_{w_j}$, a maximal parabolic subgroup of ${\bf G}_{j}$. Notice that each ${\bf M}_{j}$ corresponds to a simple root $\alpha_{j}$ of ${\bf G}_{j}$. Given a representation $\pi_0$ of $M_{\theta_0}$, we obtain a representation $\pi_j$ of $M_j$, namely,
\begin{equation}\label{def:pi:j}
   \pi_j = \pi_0 \circ {\rm Ad}(\tilde{w}_1^{-1} \cdots \tilde{w}_{j-1}^{-1}).
\end{equation}
We let
\begin{equation*}
   \delta_{\alpha_j}: M_j \xrightarrow{\delta_j^{\,l_{\!j}}} {\rm GL}_1(F) 
   \xrightarrow{\left| \cdot \right|_F} \mathbb{C}^\times, \ 
   l_j = \left\langle \delta_j, \alpha_j^\vee \right\rangle^{-1},
\end{equation*}
for any non-trivial $\delta_j \in {\rm Hom}_F({\bf M}_j/{\bf Z}_{{\bf G}_j},{\rm GL}_1)$.

We note that each $w_{j}$ in equation~\eqref{blockdecomp} is of the form
\begin{equation}
   w_{j} = w_{l,\Omega_j} w_{l,\theta_j}.
\end{equation}
Additionally, each $w_{j}$ decomposes into a product of Weyl group elements corresponding to simple roots $\alpha \in \Delta$. While these decompositions are again not unique, the choice $\mathfrak{W}$ of representatives fixes a unique $\tilde{w}_{j}$, independently of the decomposition of $w_{j}$.

If the maximal ${\bf P} = {\bf M}{\bf N}$ corresponds to the root $\alpha$, we take $\widetilde{\alpha}$ as in \eqref{alphatilde}. For each $[\beta] \in \Sigma(\theta_0,w_0)$, we let
\begin{equation*}
   i_{[\beta]} = \left\langle \widetilde{\alpha},\beta^\vee \right\rangle,
\end{equation*}
which depends only on the block $[\beta]$.
Since $\theta_0 \subset \theta$, we have that the values of $i_{[\beta]}$ range among the integer values
\begin{equation*}
   i = \left\langle \widetilde{\alpha},\gamma^\vee \right\rangle, \ 1 \leq i \leq m_R,
\end{equation*}
where $[\gamma] \in \Sigma_{\rm red}(\theta,w_0)$. Let
\[ a_j = i_{[\beta_j]} \ \text{for} \ 1 \leq j \leq d-1. \]
The blocks in $\Sigma(\theta_0,w_0)$ proportional to $[\beta_j]$ are of the form $l[\beta_j]$, $l \geq 1$, with $i_{l[\beta_j]} = l a_j$, and $l$ is the corresponding index for the pair $({\bf G}_j,{\bf M}_j)$; this is the reason for the exponent $a_j s$ in Proposition~\ref{lcmultiplicativity} below. Let $R^{(j)} = \oplus_l r^{(j)}_l$ be the adjoint action of ${}^LM_j$ on the Lie algebra of ${}^LN_{w_j}$, as in \S~\ref{Lgrp:LSrep}. 

Given a Langlands-Shahidi representation $r$ with respect to $({\bf G},{\bf M})$, let
\begin{equation}\label{blockrootsdecomp}
   \mathfrak{S}_r = \left\{ j \mid a_j \text{ divides } i_r \text{ and } b_j[\beta_j] \in \Sigma(\theta_0,w_0) \right\}, \quad b_j = b_j(r) = i_r/a_j.
\end{equation}
Then, identifying ${}^LM_j$ with ${}^LM_{\theta_0}$ via $w_{j-1} \cdots w_1$, the restriction of $r$ to ${}^LM_{\theta_0}$ decomposes as
\[ r \vert_{{}^LM_{\theta_0}} = \bigoplus_{j \in \mathfrak{S}_r} r^{(j)}_{b_j}. \]

Let us now recall Proposition~3.2.1 of \cite{Sh1981}:

\begin{proposition}\label{lcmultiplicativity}
Let $(F,\pi,\psi) \in \mathfrak{ls}(p,{\bf G},{\bf M})$ and assume $\pi$ is obtained via parabolic induction from a representation $\pi_0$ of $M_{\theta_0}$, generic with respect to the character obtained from $\psi$,
\begin{equation*}
   \pi \hookrightarrow {\rm ind}_{P_{\theta_0} \cap M_\theta}^{M_\theta} \pi_0.
\end{equation*}
For each $j$, let $\pi_j$ be the representation of $M_j$ given by \eqref{def:pi:j}. Then, with the notation of Langlands' lemma, we have
\begin{equation*}
   C_\psi(\delta_\alpha^s,\pi,\tilde{w}_0) = \prod_j 
   C_\psi(\delta_{\alpha_j}^{a_j s},\pi_j,\tilde{w}_j).
\end{equation*}
\end{proposition}

The Weyl group element representatives of \S~\ref{fixedW} are those of Tits \cite{Ti1966}, attached to the pinning; in rank one, see \S~\ref{rankone}. In particular, if $\gamma, \gamma' \in \Delta$ and $w \in W$ satisfy $w(\gamma) = \gamma'$, then conjugation by $\tilde{w}$ carries the pinned root group ${\bf N}_\gamma$ onto ${\bf N}_{\gamma'}$, compatibly with the pinnings. For split groups this is Proposition~9.3.5 of \cite{Spr2009}, see also the proof of Proposition~2.9.11 of \cite{KaPr2023}; the quasi-split case follows from the Galois invariance of the Tits section, Fact~2.9.13 of [\emph{loc.\,cit.}], since a positive absolute root restricting to a relative simple root is itself simple. We use this through the following lemma.

\begin{lemma}\label{pinning}
Let ${\bf P} = {\bf M}{\bf N}$ and ${\bf P}' = {\bf M}'{\bf N}'$ be standard maximal parabolic subgroups of quasi-split groups $\bf G$ and ${\bf G}'$, each with a fixed pinning and corresponding systems $\mathfrak{W}$ and $\mathfrak{W}'$ as in \S~\ref{fixedW}. Let $\iota : {\bf G} \rightarrow {\bf G}'$ be an isomorphism defined over $F$ with $\iota({\bf B}) = {\bf B}'$, $\iota({\bf T}) = {\bf T}'$ and $\iota({\bf M}) = {\bf M}'$, which carries the pinning of $\bf G$ to that of ${\bf G}'$ on the simple roots of $\bf M$. Let $(F,\pi,\psi) \in \mathfrak{ls}(p,{\bf G},{\bf M})$ be $\psi_{\tilde{w}_0}$-generic. Then $\pi' = \pi \circ \iota^{-1}$ is $\psi_{\tilde{w}_0'}$-generic and
\begin{equation*}
   C_\psi(s,\pi',\tilde{w}_0') = C_\psi(s,\pi,\tilde{w}_0).
\end{equation*}
\end{lemma}
\begin{proof}
The pinning of ${\bf G}'$ and the one transported by $\iota$ differ by ${\rm Ad}(y)$, for some $y \in T'_{\rm ad}$ with $\gamma(y) = 1$ for every simple root $\gamma$ of ${\bf M}'$; in particular, $y$ centralizes ${\bf M}'$. The isomorphism ${\rm Ad}(y) \circ \iota$ carries the pinning of $\bf G$ to that of ${\bf G}'$. Hence, it carries the Weyl group element representatives, which are those of Tits, the Haar measures, which are self-dual in the pinned coordinates, and the characters $\psi$ and $\psi_{\tilde{w}_0}$, to the corresponding data for ${\bf G}'$. Transport of structure in \eqref{Whittaker:Ind} and \eqref{lslc} gives the equality of local coefficients for $\pi \circ \iota^{-1} \circ {\rm Ad}(y)^{-1}$, which is $\pi'$ since $y$ centralizes ${\bf M}'$.
\end{proof}

We now present a straightforward proof of a very useful induction step.

\begin{lemma}\label{induction}
Let $({\bf G}, {\bf M})$ be a pair consisting of a connected quasi-split reductive group and ${\bf M} = {\bf M}_\theta$ a maximal Levi subgroup. Let $R = \oplus_{i=1}^{m_R} r_i$ be the adjoint action of ${}^LM$ on ${}^L\mathfrak{n}$. For each Langlands-Shahidi representation $r$ with $i_r > 1$, there exists a pair $({\bf G}_r,{\bf M})$ such that the corresponding adjoint action of ${}^LM$ on the Lie algebra of ${}^LN_r$ decomposes as
\begin{equation*}
   R' = \bigoplus_{l=1}^{m_{R'}} r_l', \ \text{with} \ m_{R'} < m_R,  
\end{equation*}
and $r = r_1'$.
\end{lemma}
\begin{proof}
Let $i = i_r$. To construct the pair $({\bf G}_r,{\bf M})$, we look at the unipotent subgroup
\begin{equation*}
   {\bf N}_{w_0}' = \prod_{h \in \Sigma'} {\bf N}_{0,h},
\end{equation*}
obtained from a subproduct of equation~\eqref{Nblockdecomp}, where $\Sigma'$ is the resulting indexing set. Each unipotent group ${\bf N}_{0,h}$ corresponds to a block set of roots $[\beta] \in \Sigma_{\rm red}(\theta,w_0)$. We take only groups ${\bf N}_{0,h}$ in the product corresponding to a $[\beta] \in \Sigma_{\rm red}(\theta,w_0)$ such that
\begin{equation*}
   \left\langle \widetilde{\alpha},\beta^\vee \right\rangle \in i\mathbb{Z}, \ \beta \in [\beta].
\end{equation*}
We then set
\begin{equation*}
   {\bf N}_r = w_0^{-1} {\bf N}_{w_0}' w_0.
\end{equation*}
For the group ${\bf G}_r$ we take the reductive group generated by ${\bf M}$, ${\bf N}_r$ and ${\bf N}_r^-$. The pair $({\bf G}_r,{\bf M})$ has
\begin{equation*}
   R' = \bigoplus_{l=1}^{m_{R'}} r_l', \quad r_l' = r_{il}, \ m_{R'} = \left[ m_R/i \right],
\end{equation*}
since the pair $({\bf G}_r,{\bf M})$ has $\widetilde{\alpha}/i$ in place of $\widetilde{\alpha}$. In particular, $r_1' = r_i = r$ and $m_{R'} < m_R$, since $i > 1$, which gives the form of the lemma.
\end{proof}

\begin{remark}
The above lemma is present in Arthur and Shahidi's work, who use endoscopy groups. In particular, an alternative proof can be given by adapting the discussion in \S~4 of \cite{Sh1990} to our situation.
\end{remark}

\subsection{Unramified principal series and Artin $L$-functions}\label{ups} Consider a triple $(F,\pi,\psi) \in \mathfrak{ls}(p,{\bf G},{\bf M})$, where $\pi$ has an Iwahori fixed vector. From Proposition~\ref{vary}, we can assume $\pi$ is $\psi_{\tilde{w}_0}$-generic and $\tilde{w}_0$-compatible with $\psi$; the choice of Weyl group element representatives and Haar measures $\mathfrak{W}$ being fixed. With multiplicativity of the local coefficient, we can proceed as in \S~2 of \cite{KeSh1988} and \S~3 of \cite{Sh1990}, and reduce the problem to the rank one cases of \S~\ref{rankone}. We now proceed to state the main result.

For every root $\beta \in \Sigma_{\rm red}(\emptyset,w_0)$, we have as in \S~\ref{fixedW} a corresponding group ${\bf G}_\beta$ of semisimple $F$-rank one, whose derived group is ${\rm Res}_{F_\beta'/F}{\rm SL}_2$ or ${\rm Res}_{F_\beta'/F}{\rm SU}_3$, the latter with respect to a quadratic extension $E_\beta/F_\beta'$. Let $\Sigma_{\rm red}(w_0,{\rm SL}_2)$ denote the set consisting of $\beta \in \Sigma_{\rm red}(\emptyset,w_0)$ such that ${\bf G}_\beta$ is as in case~(i) of Proposition~\ref{rankoneprop}. Similarly, let $\Sigma_{\rm red}(w_0,{\rm SU}_3)$ consist of $\beta \in \Sigma_{\rm red}(\emptyset,w_0)$ such that ${\bf G}_\beta$ is as in the corresponding case~(ii). Let
\begin{equation*}
   \lambda(\psi,w_0) = \prod_{\beta \in \Sigma_{\rm red}(w_0,{\rm SL}_2)} \lambda(F_\beta'/F,\psi)
   				   \prod_{\beta \in \Sigma_{\rm red}(w_0,{\rm SU}_3)} \lambda(E_\beta/F,\psi)^2 \lambda(F_\beta'/F,\psi)^{-1}.
\end{equation*}
For $\beta \in \Sigma_{\rm red}(\emptyset,w_0)$, let $i_\beta = \langle \widetilde{\alpha},\beta^\vee \rangle$, computed with any absolute root restricting to $\beta$; if $\beta \in \Sigma_{\rm red}(w_0,{\rm SU}_3)$, the absolute roots restricting to $2\beta$ have index $2i_\beta$.

Let $r$ be a Langlands-Shahidi representation for the pair $({\bf G},{\bf M})$. We let
\begin{multline*}
   \lambda_r(\psi,w_0) = \prod_{\substack{\beta \in \Sigma_{\rm red}(w_0,{\rm SL}_2) \\ i_\beta = i_r}} \lambda(F_\beta'/F,\psi) \\
   \times \prod_{\substack{\beta \in \Sigma_{\rm red}(w_0,{\rm SU}_3) \\ i_\beta = i_r}} \lambda(E_\beta/F,\psi)^2
   \prod_{\substack{\beta \in \Sigma_{\rm red}(w_0,{\rm SU}_3) \\ 2i_\beta = i_r}} \lambda(F_\beta'/F,\psi)^{-1},
\end{multline*}
so that, with $R = \oplus\, r_i$, we have
\[ \lambda(\psi,w_0) = \prod_{i=1}^{m_R} \lambda_{r_i}(\psi,w_0). \]
In rank one, this agrees with the normalization of $\gamma$-factors in \S~\ref{rankone}.

It is worth noting that $\lambda$-factors are trivial when $\bf G$ is split. Now, using the following two properties of $\lambda$-factors,
\[ \lambda(E/F,\psi) = \lambda(E/F',\psi_{F'}) \lambda(F'/F,\psi)^2 \text{ and } \lambda(E/F',\overline{\psi}_{F'}) \lambda(E/F',\psi_{F'}) = 1, \]
Proposition~\ref{rankoneprop} and the normalization of $\gamma$-factors in \S~\ref{rankone} give the rank one case of the following proposition.

\begin{proposition}\label{unrgammaprop}
Let $(F,\pi,\psi) \in \mathfrak{ls}(p,{\bf G},{\bf M})$ be such that $\pi$ has an Iwahori fixed vector. Then
\begin{equation*}
   C_\psi(s,\pi,\tilde{w}_0) = \lambda(\psi,w_0)^{-1} \prod_{i=1}^{m_R} \gamma(is,\pi,r_i,\psi).
\end{equation*}
Let $\phi: \mathcal{W}'_F \rightarrow {}^LM$ be the parameter of the Weil-Deligne group corresponding to $\pi$. Then
\begin{equation*}
   \prod_{i=1}^{m_R} \gamma(is,\pi,r_i,\psi) = \prod_{i=1}^{m_R} \gamma(is,r_i \circ \phi,\psi),
\end{equation*}
where on the right hand side we have the Artin $\gamma$-factors defined by Deligne and Langlands \cite{Ta1979}.
\end{proposition}

\section{A local to global result}\label{lg}

We provide a simple proof of a local to global result, Lemma~\ref{hllemma}, which is the quasi-split reductive groups generalization of the local to global result of Henniart-Lomel\'i for ${\rm GL}_n$ over function fields \cite{HeLo2013a}. A general globalization theorem in characteristic $p$ is proved by Gan-Lomel\'i \cite{GaLoJEMS}. 
We now minimize the number of auxiliary places required to one.

\subsection{} For the remainder of the article, $k$ will denote a global function field with field of constants $\mathbb{F}_q$ and ring of ad\`eles $\mathbb{A}_k$. Given an algebraic group $\bf H$ and a place $v$ of $k$, we write $q_v$ for $q_{k_v}$ and $H_v$ instead of ${\bf H}(k_v)$. Similarly with $\mathcal{O}_v$ and $\varpi_v$.

We consider pairs $({\bf G},{\bf M})$ of quasi-split connected reductive group schemes defined over $k$. Globally we fix a maximal compact open subgroup $\mathcal{K} = \prod_v \mathcal{K}_v$ of ${\bf G}(\mathbb{A}_k)$, where the $\mathcal{K}_v$ range through a fixed set of maximal compact open subgroups of $G_v$. Each $\mathcal{K}_v$ is special and $\mathcal{K}_v $ is hyperspecial at almost every place. In addition, we can choose each $\mathcal{K}_v$ to be compatible with the decomposition
\begin{equation*}
   \mathcal{K}_v = (N_v^- \cap \mathcal{K}_v) (M_v \cap \mathcal{K}_v)(N_v \cap \mathcal{K}_v),
\end{equation*}
for every standard parabolic subgroup ${\bf P} = {\bf M}{\bf N}$. The group $M \cap \mathcal{K}$ is a maximal compact open subgroup of $M = {\bf M}(\mathbb{A}_k)$. Furthermore $G = P \mathcal{K}$.

\begin{lemma}\label{hllemma} Let $v_0$ be a place of $k$, and let $\pi_0$ be a supercuspidal representation of $G_{v_0}$. There is a place $v_\infty \neq v_0$ of $k$ and a cuspidal automorphic representation $\pi = \otimes_v \pi_v$ of ${\bf G}(\mathbb{A}_k)$ satisfying the following properties:
\begin{itemize}
   \item[(i)] $\pi_{v_0} \cong \pi_{0}$;
   \item[(ii)] $\pi_{v}$ has an Iwahori fixed vector for $v \notin \left\{ v_0, v_\infty \right\}$;
   \item[(iii)] $\pi_{v_\infty}$ is a constituent of a tamely ramified principal series;
   \item[(iv)] if $\psi$ is a non-degenerate character of ${\bf U}(k) \backslash {\bf U}(\mathbb{A}_k)$ and $\pi_0$ is $\psi_{v_0}$-generic, then $\pi$ is globally $\psi$-generic.
\end{itemize}
\end{lemma}

\begin{proof}
   We construct a function $f = \otimes_v f_v$ on ${\bf G}(\mathbb{A}_k)$ in such a way that the Poincar\'e series
\begin{equation*}
   {\rm P}f(g) = \sum_{\gamma \in {\bf G}(k)} f(\gamma g)
\end{equation*}
is a cuspidal automorphic function.

Let $\bf Z$ be the center of $\bf G$. There is a positive unramified character $\nu$ of $G_{v_0}$ such that $\pi_0 \otimes \nu^{-1}$ has unitary central character, and $\nu$ extends to a character of ${\bf G}(k) \backslash {\bf G}(\mathbb{A}_k)$ which is unramified at every place; twisting the resulting global representation by it, we may assume that $\omega_{\pi_0}$ is unitary. It is possible to construct a global unitary character $\omega : {\bf Z}(k) \backslash {\bf Z}(\mathbb{A}_k) \rightarrow \mathbb{C}^\times$ such that: $\omega_{v_0} = \omega_{\pi_0}$; $\omega_{v_\infty}$ is trivial on $Z_{v_\infty}^1$; and, $\omega_v$ is trivial on $Z_v^0$ for every $v \notin \left\{ v_0, v_\infty \right\}$. Here, $Z_{v_\infty}^1$ is the pro-$p$ unipotent radical ${\rm mod}\,\mathfrak{p}_{v_\infty}$ of $Z_{v_\infty}$ and $Z_v^0$ is the maximal compact open subgroup of $Z_v$.

We let $S' = \left\{ v_0, v', v_\infty \right\}$, where $v'$ is an auxiliary place of $k$. Outside of $S'$ we consider the characteristic functions
\begin{equation*}
   f_v = \mathds{1}_{\mathcal{K}_v}, \ v \notin S'.
\end{equation*}
At the place $v_0$ of $k$, we let
\begin{equation*}
   f_{v_0}(g) = \left\langle \pi_0(g) \xi, \eta \right\rangle, \ \xi, \eta \in V_{\pi_0},
\end{equation*}
be a matrix coefficient of $\pi_0$, where $V_{\pi_0}$ is the space of $\pi_0$. The function $f_{v_0}$ has compact support $\mathcal{C}_{v_0}$ modulo $Z_{v_0}$.

The places $v'$ and $v_\infty$ are chosen so that $\mathcal{K}_{v'}$ and $\mathcal{K}_{v_\infty}$ are hyperspecial. At the place $v'$, let $\mathcal{I}_{v'}$ be the Iwahori subgroup attached to $\bf B$, i.e., the inverse image in $\mathcal{K}_{v'}$ of the reduction ${\rm mod} \, \mathfrak{p}_{v'}$ of $\bf B$. At the place $v_\infty$, let $\mathcal{I}_{v_\infty}^1$ be the pro-$p$ Iwahori subgroup attached to the opposite Borel subgroup ${\bf B}^-$, i.e., the inverse image in $\mathcal{K}_{v_\infty}$ of the reduction ${\rm mod} \, \mathfrak{p}_{v_\infty}$ of the unipotent radical ${\bf U}^-$ of ${\bf B}^-$. Let
\begin{equation*}
   f_{v'} = \mathds{1}_{\mathcal{I}_{v'}}   \text{ and }   f_{v_\infty} = \mathds{1}_{\mathcal{I}_{v_\infty}^1}.
\end{equation*}
Twist the global function $f = \otimes f_v$ by $\omega^{-1}$, in order to be able to mod out by the center.

Let $\overline{\bf H} = {\bf G}/{\bf Z}$. We fix a faithful representation $\overline{\bf H} \hookrightarrow {\rm GL}_n$ over $k$ taking the images of $\bf B$ and ${\bf B}^-$ into upper and lower triangular matrices, respectively; for almost every place $v$ it is defined over $\mathcal{O}_v$ and compatible with $\mathcal{K}_v$, and we choose $v'$ and $v_\infty$ among these places. The entries $g_{ij}$ of an element $g$ of $\overline{\bf H}$ refer to this embedding.

We now set
\begin{equation*}
   \mathcal{C} = \mathcal{C}_{v_0} \times \prod_{v \notin S'} \mathcal{K}_v \times 
   \mathcal{I}_{v'} \times \mathcal{I}_{v_\infty}^1.
\end{equation*}
The projection $\overline{\mathcal{C}}$ of $\mathcal{C}$ on $\overline{\bf H}(\mathbb{A}_k)$ is compact. Hence, $\overline{\mathcal{C}} \cap \overline{\bf H}(k)$ is finite. Because of this, there exists a constant $h$, such that the height $\lVert g \rVert \leq h$ bounds the entries of $g = (g_{ij}) \in \overline{\bf H}(k) \cap\overline{\mathcal{C}}$.  See \S~I.2.2 of \cite{MoWa1994} for the height functions $\lVert g \rVert$ and $\lVert g \rVert_v$ for elements of ${\bf G}(\mathbb{A}_k)$ and $G_v$, respectively.

We now impose the further conditions on the choice of places $v'$ and $v_\infty$. These two places need to be such that the cardinality of their respective residue fields is larger than $h$. This is to ensure that the poles of the entries of all $g \in \overline{\bf H}(k) \cap \overline{\mathcal{C}}$ are absorbed. Indeed, given a connected smooth projective curve $X$ over $\mathbb{F}_q$, its set of points over the algebraic closure $\bar{\mathbb{F}}_q$ is infinite. Hence, it is always possible to find places $v$ with residue field $\mathbb{F}_{q_v}$, with $q_v > h$.

We then have that
\begin{equation*}
   g_{ij} \in \mathcal{O}_{v'}, \ i \leq j,
\end{equation*}
and the congruence relations
\begin{equation*}
   g_{ij} \equiv 0 \text{ mod } \mathfrak{p}_{v'}, \ i > j.
\end{equation*}
By incorporating the condition at infinity, we obtain similar relationships corresponding to the use of lower triangular matrices at $v_\infty$. We can see that $\overline{\mathcal{C}} \cap \overline{\bf H}(k) \subset \overline{\bf H}(\mathbb{F}_q)$. In fact, the pro-$p$ condition at $v_\infty$ ensures that $\overline{\mathcal{C}} \cap \overline{\bf H}(k) = \left\{ I_n \right\}$. Now, we incorporate the twist by $\omega$, and lift the function $f = \otimes f_v$ back to one of ${\bf G}(\mathbb{A}_k)$. We then have that
\begin{equation*}
   {\rm P}f(g) = f(g), \text{ for } g \in \mathcal{C}.
\end{equation*}

We can now proceed as in p.~4033 of \cite{HeLo2013a} to conclude that ${\rm P}f$ belongs to the space $L_0^2({\bf G},\omega)$ of cuspidal automorphic functions on ${\bf G}(k) \backslash {\bf G}(\mathbb{A}_k)$ transforming via $\omega$ under ${\bf Z}(\mathbb{A}_k)$. The resulting cuspidal automorphic representation $\pi$ of ${\bf G}(\mathbb{A}_k)$ is such that: $\pi_{v_0} \cong \pi_0$; $\pi_{v'}$ has a non-zero fixed vector under $\mathcal{I}_{v'}$; and, $\pi_{v_\infty}$ has a non-zero fixed vector under $\mathcal{I}_{v_\infty}^1$. The last three paragraphs of the proof of Theorem~3.1 of [\emph{loc.\,cit.}] are general and can be used to establish property~(iii) of our lemma.

For property~(iv), let $\psi$ be a non-degenerate character of ${\bf U}(k) \backslash {\bf U}(\mathbb{A}_k)$, with local component $\psi_{v_0}$ on $U_{v_0}$. Since $\pi_0$ is supercuspidal and $\psi_{v_0}$-generic, we can choose the matrix coefficient $f_{v_0}$ so that its local Whittaker integral $\int_{U_{v_0}} f_{v_0}(n) \overline{\psi}_{v_0}(n) \, dn$ is non-zero. We can then proceed as in Theorem~2.2 of \cite{Vi1986} to show that
\begin{equation*}
   W_\psi({\rm P}f)(g) = \int_{{\bf U}(k) \backslash {\bf U}(\mathbb{A}_k)} {\rm P}f(ng) \overline{\psi}(n) \, dn \neq 0.
\end{equation*}
Hence, the Whittaker model is globalized in this construction. Such a $\psi$ always exists. Indeed, let $\varphi = \otimes_v \varphi_v : k \backslash \mathbb{A}_k \rightarrow \mathbb{C}^\times$ be a non-trivial character, extended to ${\bf U}(k) \backslash {\bf U}(\mathbb{A}_k)$ via the pinning of the roots. In general, $\pi_0$ need not be $\varphi_{v_0}$-generic. However, ${\bf T}_{\rm ad}$, the image of $\bf T$ in the adjoint group of $\bf G$, is an induced torus, since the simple roots form a Galois stable basis of its character group; hence ${\bf T}_{\rm ad}(k)$ is dense in $T_{{\rm ad},v_0}$ by weak approximation. By Lemma~\ref{generic:open} below, there is $y \in {\bf T}_{\rm ad}(k)$ such that $\pi_0$ is $\varphi_{v_0}^y$-generic, and we may take $\psi = \varphi^y$.

\end{proof}

\begin{lemma}\label{generic:open}
Let ${\bf H}$ be a quasi-split connected reductive group over a non-arch\-i\-me\-de\-an local field $F$, with Borel subgroup ${\bf B} = {\bf T}{\bf U}$, and let ${\bf T}_{\rm ad}$ be the image of $\bf T$ in the adjoint group of $\bf H$. Let $\pi_0$ be a supercuspidal representation of $H$ which is $\chi$-generic for a non-degenerate character $\chi$ of $U$. Then there is an open subgroup $A$ of $T_{\rm ad}$ such that $\pi_0$ is $\chi^y$-generic for every $y \in A$, where $\chi^y(u) = \chi(y^{-1}uy)$.
\end{lemma}
\begin{proof}
For $y \in T_{\rm ad}$, let $\pi_0^y = \pi_0 \circ {\rm Ad}(y^{-1})$. A $\chi$-Whittaker functional for $\pi_0$ is a $\chi^y$-Whittaker functional for $\pi_0^y$, hence it suffices to show that $\pi_0^y \cong \pi_0$ for $y$ in an open subgroup of $T_{\rm ad}$. Let $T_{\rm ad}^0$ be the maximal compact subgroup of $T_{\rm ad}$, and choose a compact open subgroup $K$ of $H$, normalized by $T_{\rm ad}^0$, with $\pi_0^K \neq 0$. Since $\pi_0$ is supercuspidal and $\pi_0^K$ is finite dimensional, the operators $\pi_0(\mathds{1}_{KhK})$ on $\pi_0^K$ vanish outside finitely many double cosets $KhK$ modulo $Z_H$. There is an open subgroup $A$ of $T_{\rm ad}^0$ fixing each of these double cosets. For $y \in A$, ${\rm Ad}(y)$ permutes the double cosets and is trivial on $Z_H$, so the Hecke algebra $\mathcal{H}(H, K)$ acts in the same way on $\pi_0^K$ and on $(\pi_0^y)^K = \pi_0^K$. Hence $\pi_0^y \cong \pi_0$.
\end{proof}

\begin{remark}
In characteristic $p$, the image of $T_{v_0}$ in $T_{{\rm ad},v_0}$ need not be open when the center of $\bf G$ is not smooth, e.g., for ${\rm SL}_2$ in characteristic $2$. This is why, in Lemma~\ref{generic:open}, we work with the stabilizer of $\pi_0$ in $T_{\rm ad}$ rather than with the image of $T$.
\end{remark}

\begin{remark}
Lemma~\ref{hllemma} admits a modification. Property \emph{(ii)} can be replaced by: \emph{(ii)'} $\pi_{v}$ is spherical for $v \notin \left\{ v_0, v_\infty \right\}$; and, condition \emph{(iii)} replaced by: \emph{(iii)'} $\pi_{v_\infty}$ is a level zero supercuspidal representation in the sense of Morris \cite{Mo1999}. If we begin with a level zero supercuspidal $\pi_0$, the above globalization produces a cuspidal automorphic representation $\pi$ such that $\pi_v$ arises from an unramified principal series at every $v \notin \left\{ v_0, v_\infty \right\}$.
\end{remark}

\section{Partial $L$-functions and the local coefficient}\label{partialL}

The adjoint representation of the group ${}^LM$ on ${}^L\mathfrak{n}$ decomposes into irreducible constituents $r_i$, $1 \leq i \leq m_R$, as in \S~\ref{Lgrp:LSrep}. Locally, let $(F,\pi,\psi) \in \mathfrak{ls}(p)$ be such that $\bf M$ is unramified over $F$ and $\pi$ is $K_M$-unramified, for a hyperspecial maximal compact subgroup $K_M$ of $M$. Then $\pi = \pi(\chi)$ for a character $\chi$ of $T$ trivial on $T \cap K_M$, and $\chi$ corresponds to an unramified homomorphism $\sigma(\chi) : \mathcal{W}_F \rightarrow {}^LT$; composing with the inclusion $\iota_M : {}^LT \hookrightarrow {}^LM$ gives the Langlands parameter of $\pi$. For each $i$, we then have
\begin{equation}\label{eq:L:nr:def}
   L(s,\pi,r_i) = L(s, r_i \circ \iota_M \circ \sigma(\chi))
      = \dfrac{1}{{\rm det}(I - r_i(\iota_M \circ \sigma(\chi))({\rm Fr}) \, q_F^{-s})},
\end{equation}
where ${\rm Fr}$ is a geometric Frobenius element.

\subsection{Global notation}\label{globalnot}
Let $\mathcal{LS}(p,{\bf G},{\bf M})$, or simply $\mathcal{LS}(p)$, be the class of quadruples $(k,\pi,\psi,S)$ consisting of: a global function field $k$ of characteristic $p$; a globally generic cuspidal automorphic representation $\pi = \otimes_v \pi_v$ of ${\bf M}(\mathbb{A}_k)$; a non-trivial character $\psi = \otimes_v \psi_v: k \backslash \mathbb{A}_k \rightarrow \mathbb{C}^\times$; and, a finite set of places $S$ such that, for every $v \notin S$, the group ${\bf G}$ is unramified over $k_v$, $\mathcal{K}_v$ is hyperspecial, $\pi_v$ is $(\mathcal{K}_v \cap M_v)$-unramified, and $\psi_v$ is unramified.

Given $(k,\pi,\psi,S) \in \mathcal{LS}(p)$, we have partial $L$-functions
\begin{equation*}
   L^S(s,\pi,r_i) = \prod_{v \notin S} L(s,\pi_v,r_{i,v}).
\end{equation*}
They are absolutely convergent for $\Re(s) \gg 0$ \cite{Bo1979}.

Fix $(k,\pi,\psi,S) \in \mathcal{LS}(p)$. From the discussion of \S~\ref{fixedW}, the character $\psi = \otimes_v \psi_v$ gives a self-dual Haar measure $d\mu_v$ at every place $v$ of $k$. We let
\begin{equation*}
   d\mu = \prod_v d\mu_v.
\end{equation*}
Notice that $d\mu_v(\mathcal{O}_v) = d\mu_v^\times(\mathcal{O}_v^\times) = 1$ for all $v \notin S$. Representatives of Weyl group elements are chosen using Langlands' lemma for ${\bf G}(k)$. This globally fixes the system
\begin{equation}\label{globWsystem}
   \mathfrak{W} = \left\{ \tilde{w}_\alpha, d\mu_\alpha \right\}_{\alpha \in \Delta}.
\end{equation}
As in \S~\ref{fixedW}, this fixes the Haar measure on ${\bf N}(\mathbb{A}_k)$.

We obtain a character of ${\bf U}(\mathbb{A}_k)$ from the fixed character $\psi$ and pinning of the roots, which we can again denote by $\psi$. Given a global non-degenerate character $\chi$ of ${\bf U}(k) \backslash {\bf U}(\mathbb{A}_k)$, we obtain a global character $\chi_{\tilde{w}_0}$ of ${\bf U}_M(\mathbb{A}_k) = {\bf M}(\mathbb{A}_k) \cap {\bf U}(\mathbb{A}_k)$ which is $\tilde{w}_0$-compatible with $\chi$, in the sense that the following relationship holds
\begin{equation}\label{glob:w0:compatible}
   \chi_{\tilde{w}_0}(u) = \chi(\tilde{w}_0 u \tilde{w}_0^{-1}), \quad u \in {\bf U}_M(\mathbb{A}_k).
\end{equation}
In particular, we obtain $\tilde{w}_0$-compatible characters $\psi$ and $\psi_{\tilde{w}_0}$.

We note that the discussion of Appendix~A of \cite{CoKiPSSh2004} is valid also for the case of function fields. In particular, Lemma~A.1 of [\emph{loc.\,cit.}] combined with Proposition~5.4 of \cite{Sh2002} gives Proposition~\ref{globalvary} below, which allows us to address the variance of the globally generic character.

\begin{proposition}\label{globalvary}
Let $(k,\pi,\psi,S) \in \mathcal{LS}(p)$. There exists a connected quasi-split reductive group $\widetilde{\bf G}$ defined over $k$, sharing the same derived group as $\bf G$, whose center is connected with trivial $H^1$, and with maximal torus $\widetilde{\bf T} = {\bf Z}_{\widetilde{\bf G}} {\bf T}$. Then, there exists an element $x \in \widetilde{\bf T}(k)$ such that the representation $\pi_x$, given by
\begin{equation*}
   \pi_x(g) = \pi(x^{-1}gx)
\end{equation*}
is $\psi_{\tilde{w}_0}$-generic. Furthermore, after enlarging $S$ if necessary, we have equality of partial $L$-functions
\begin{equation*}
   L^S(s,\pi,r_i) = L^S(s,\pi_x,r_i).
\end{equation*}
\end{proposition}

\subsection{Crude functional equation and rationality}\label{crude:rat}
Eisenstein series provide the link between the Langlands-Shahidi local coefficient and automorphic $L$-functions. We first register a functional equation in Theorem~\ref{crudeFE}, that is crude in the sense that it pertains to a product of $L$-functions over all of the $r_i$'s combined. In the next section, we refine our results in order to obtain a functional equation for individual $L$-functions and local factors.

\begin{theorem} \label{crudeFE}
   Let $(k,\pi,\psi,S) \in \mathcal{LS}(p)$ with $\pi$ $\psi_{\tilde{w}_0}$-generic, and let $\oplus_{i=1}^{m_R} \, r_i$ be the adjoint action of ${}^LM$ on ${}^L\mathfrak{n}$. Then
   \begin{equation*}
      \prod_{i=1}^{m_R} L^S(is,\pi,r_i) = \prod_{v \in S} C_{\psi}(s,\pi_v,\tilde{w}_0) \prod_{i=1}^{m_R} L^S(1-is,\tilde{\pi},r_i).
   \end{equation*}
Furthermore, each partial $L$-function $L^S(s,\pi,r_i)$ is a rational function of the variable $q^{-s}$.
\end{theorem}

The proof of the crude functional equation and rationality over function fields found in \cite{Lo2016} for Siegel Levi subgroups of the quasi-split classical groups generalizes to all Langlands-Shahidi pairs $({\bf G},{\bf M})$. The proof of rationality for completed $L$-functions, which includes the case of partial $L$-functions, can be found in \cite{Lo2019}, in addition to further connections to Eisenstein series.

\section{The Langlands-Shahidi method over function fields}\label{LS:main}

We now come to the main result of the Langlands-Shahidi method over function fields and applications. Our discussion is for pairs $({\bf G},{\bf M})$ of quasi-split connected reductive groups such that $\bf M$ is a Levi component for a maximal parabolic subgroup $\bf P$ of $\bf G$ over a global function field $k$. As in the local case, a \emph{Langlands-Shahidi representation} is an irreducible constituent $r$ of the adjoint action of ${}^LM$ on the Lie algebra of the unipotent radical of ${}^LP$.

\subsection{Main theorem}
The following theorem over number fields is due to Shahidi \cite{Sh1990}. We note that compatibility with Artin factors for real groups is the subject of \cite{Sh1985}. We now focus on the new characteristic $p$ case.

\begin{theorem}\label{mainthm}
Given the pair $({\bf G},{\bf M})$, let $r$ be a Langlands-Shahidi representation. There exists a system of rational $\gamma$-factors, $L$-functions and $\varepsilon$-factors on $\mathfrak{ls}(p)$, uniquely determined by the following properties:
\begin{enumerate}
\item[(i)] \emph{(Naturality).} Let $(F,\pi,\psi) \in \mathfrak{ls}(p)$. Let $\eta: F' \rightarrow F$ be an isomorphism of non-archimedean local fields and let $(F',\pi',\psi') \in \mathfrak{ls}(p)$ be the triple obtained via $\eta$. Then
   \begin{equation*}
      \gamma(s,\pi,r,\psi) = \gamma(s,\pi',r,\psi').
   \end{equation*}
\item[(ii)] \emph{(Isomorphism)}. Let $(F,\pi_j,\psi) \in \mathfrak{ls}(p)$, $j=1$, $2$. If $\pi_1 \cong \pi_2$, then
   \begin{equation*}
      \gamma(s,\pi_1,r,\psi) = \gamma(s,\pi_2,r,\psi).
   \end{equation*}
\item[(iii)] \emph{(Compatibility with Artin factors).} Let $(F,\pi,\psi) \in \mathfrak{ls}(p)$ be such that $\pi$ has an Iwahori fixed vector. Let $\sigma : \mathcal{W}'_F \rightarrow {}^LM$ be the Langlands parameter corresponding to $\pi$. Then
   \begin{equation*}
      \gamma(s,\pi,r,\psi) = \gamma(s,r \circ \sigma,\psi).
   \end{equation*} 
\item[(iv)] \emph{(Multiplicativity).} Let $(F,\pi,\psi) \in \mathfrak{ls}(p)$ be such that
   \begin{equation*}
      \pi \hookrightarrow {\rm ind}_{P_{\theta_0} \cap M}^{M} (\pi_0),
   \end{equation*}
where $\pi_0$ is a generic representation of $M_{\theta_0}$, with $\theta_0 \subset \theta$. The restriction of $r$ to ${}^LM_{\theta_0}$ decomposes as a sum of irreducibles $r^{(j)}_{b_j}$, $j$ in an indexing set $\mathfrak{S}_r$, as in \eqref{blockrootsdecomp}. With $(F,\pi_{j},\psi) \in \mathfrak{ls}(p,{\bf G}_j,{\bf M}_j)$, obtained from Proposition~\ref{lcmultiplicativity}, we have
\begin{equation*}
   \gamma(s,\pi,r,\psi) = \prod_{j \in \mathfrak{S}_r} \gamma(s,\pi_{j},r^{(j)}_{b_j},\psi).
\end{equation*}
\item[(v)] \emph{(Dependence on $\psi$).} Let $(F,\pi,\psi) \in \mathfrak{ls}(p)$. For $a \in F^\times$, let $\psi^a : F \rightarrow \mathbb{C}^\times$ be the character given by $\psi^a(x) = \psi(ax)$. Then, there exists $z(a) \in Z_M$ such that
   \begin{equation*}
      \gamma(s,\pi,r,\psi^a) = \kappa_r(a) \, \omega_{\pi}(z(a)) \left| a \right|_F^{n_r(s-\frac{1}{2})} \cdot \gamma(s,\pi,r,\psi),
   \end{equation*}
   where $n_r = \dim {}^L \mathfrak{n}_r$ and $\kappa_r(a) = \lambda_r(\psi^a,w_0)/\lambda_r(\psi,w_0) \in \{ \pm 1 \}$.
\item[(vi)] \emph{(Functional equation).} Let $(k,\pi,\psi,S) \in \mathcal{LS}(p)$. Then
   \begin{equation*}
      L^S(s,\pi,r) = \prod_{v \in S} \gamma(s,\pi_v,r_v,\psi_v) L^S(1-s,\tilde{\pi},r).
   \end{equation*}
\item[(vii)] \emph{(Tempered $L$-functions).} For $(F,\pi,\psi) \in \mathfrak{ls}(p)$ tempered, let $P_{\pi,r}(t)$ be the unique polynomial with $P_{\pi,r}(0) = 1$ and such that $P_{\pi,r}(q_F^{-s})$ is the numerator of $\gamma(s,\pi,r,\psi)$. Then
   \begin{equation*}
      L(s,\pi,r) = \dfrac{1}{P_{\pi,r}(q_F^{-s})}
   \end{equation*}
is holomorphic and non-zero for $\Re(s) > 0$.
\item[(viii)] \emph{(Tempered $\varepsilon$-factors).} Let $(F,\pi,\psi) \in \mathfrak{ls}(p)$ be tempered. Then
   \begin{equation*}
      \varepsilon(s,\pi,r,\psi) = \gamma(s,\pi,r,\psi) \dfrac{L(s,\pi,r)}{L(1-s,\tilde{\pi},r)}
   \end{equation*}
is a monomial in $q_F^{-s}$.
\item[(ix)] \emph{(Twists by unramified characters).} Let $(F,\pi,\psi) \in \mathfrak{ls}(p)$ and $s_0 \in \mathbb{C}$. Then
   \begin{align*}
      \gamma(s,\delta_\alpha^{s_0} \otimes \pi,r,\psi) &= \gamma(s+i_rs_0,\pi,r,\psi), \\
      L(s,\delta_\alpha^{s_0} \otimes \pi,r) &= L(s+i_rs_0,\pi,r), \\
      \varepsilon(s,\delta_\alpha^{s_0} \otimes \pi,r,\psi) &= \varepsilon(s+i_rs_0,\pi,r,\psi).
   \end{align*}
\item[(x)] \emph{(Langlands' classification).} Let $(F,\pi,\psi) \in \mathfrak{ls}(p)$. Let $\pi_0$ be a tempered generic representation of $M_0 = M_{\theta_0}$ and $\chi$ a character of $A_0 = A_{\theta_0}$ which is in the Langlands situation. Suppose $\pi$ is the Langlands quotient of the representation
\begin{equation*}
   \xi = {\rm Ind}(\pi_{0,\chi}),
\end{equation*}
with $\pi_{0,\chi} = \pi_0 \cdot \chi$. Suppose $(F,\pi_j,\psi) \in \mathfrak{ls}(p,{\bf G}_j,{\bf M}_j)$ are quasi-tempered, where the $\pi_j$ are obtained via Langlands' lemma and equation~\eqref{def:pi:j} applied to $\pi_{0,\chi}$. Then
\begin{align*}
   L(s,\pi,r) &= \prod_{j \in \mathfrak{S}_{r}} L(s,\pi_j,r^{(j)}_{b_j}), \\
   \varepsilon(s,\pi,r,\psi) &= \prod_{j \in \mathfrak{S}_{r}} \varepsilon(s,\pi_j,r^{(j)}_{b_j},\psi),
\end{align*}
where the restriction of $r$ to ${}^LM_{\theta_0}$ decomposes as a sum of irreducibles $r^{(j)}_{b_j}$, $j$ in the indexing set $\mathfrak{S}_r$ of \eqref{blockrootsdecomp}.
\end{enumerate}

\end{theorem}

\begin{corollary} Langlands-Shahidi $\gamma$-factors, $L$-functions and root numbers satisfy:
\begin{enumerate}
   \item[(xi)] (Local functional equation) \emph{Let $(F,\pi,\psi) \in \mathfrak{ls}(p)$, then
   \begin{equation*}
      \gamma(s,\pi,r,\psi) \gamma(1-s,\tilde{\pi},r,\overline{\psi}) = 1.
   \end{equation*}
   }
   \item[(xii)] (Global functional equation) \emph{Let $(k,\pi,\psi,S) \in \mathcal{LS}(p)$, then
\begin{equation*}
   L(s,\pi,r) = \varepsilon(s,\pi,r) L(1-s,\tilde{\pi},r),
\end{equation*}
where
\begin{equation*}
   L(s,\pi,r) = \prod_v L(s,\pi_v,r_v) \text{ and } \varepsilon(s,\pi,r) = \prod_v \varepsilon(s,\pi_v,r_v,\psi_v).
\end{equation*}
   }
\end{enumerate}
\end{corollary}
\begin{proof}
We formally deduce the local functional equation for $\gamma$-factors from Theorem~\ref{mainthm}. Indeed, as observed in \S~2.5 of \cite{HeLo2013a} for the case of ${\rm GL}_n$, we can apply uniqueness to a system of $\gamma$-factors $\gamma'$ defined by $\gamma(1-s,\tilde{\pi},r,\overline{\psi})^{-1}$. Notice that we only use Properties~(i)--(vi) to prove Property~(xi).

Globally, the functional equation for partial $L$-functions, namely Property~(vi) of Theorem~\ref{mainthm}, together with the existence of a system of $\gamma$-factors, $L$-functions and root numbers, and the agreement of $L$-functions at unramified places with \eqref{eq:L:nr:def}, see \cite{dCaHeLo}, gives the functional equation for completed Langlands-Shahidi automorphic $L$-functions.
\end{proof}

\subsection{Proof of Theorem~\ref{mainthm}} The crude functional equation of Theorem~\ref{crudeFE}, together with Proposition~\ref{unrgammaprop}, points us towards the existence part of Theorem~\ref{mainthm} concerning $\gamma$-factors. Indeed, let $(F,\pi,\psi) \in \mathfrak{ls}(p)$ be such that $\pi$ is $\psi_{\tilde{w}_0}$-generic. Then we recursively define $\gamma$-factors via the equation
\begin{equation}\label{recursivedef}
   C_\psi(s,\pi,\tilde{w}_0) = \prod_{i=1}^{m_R} \lambda_{r_i}(\psi,w_0)^{-1} \gamma(is,\pi,r_i,\psi)
\end{equation}
and Lemma~\ref{induction}. For arbitrary $(F,\pi,\psi) \in \mathfrak{ls}(p)$, they are defined with the aid of Proposition~\ref{vary}: we choose $x \in \widetilde{T}$ such that $\pi_x$ is $\psi_{\tilde{w}_0}$-generic, and define the $\gamma$-factors of $\pi$ to be those of $\pi_x$. If $x' \in \widetilde{T}$ is another such element, then $\pi_{x'} \cong \pi_x$ by Proposition~\ref{x:independence}, so the definition does not depend on the choice of $x$. From Theorem~\ref{lc:rationality}, it follows that $\gamma(s,\pi,r_i,\psi) \in \mathbb{C}(q_F^{-s})$.

With the above definition of $\gamma$-factors based on equation~\eqref{recursivedef}, Properties~(i) and (ii) can be readily verified. The inductive argument on the adjoint action together with Proposition~\ref{unrgammaprop} give Property~(iii). Property~(iv) is a consequence of the multiplicativity of the local coefficient, Proposition~\ref{lcmultiplicativity}, as we now explain.

Let $\pi \hookrightarrow {\rm ind}_{P_{\theta_0} \cap M}^M(\pi_0)$ be as in Property~(iv). By Proposition~\ref{vary} and the definition of $\gamma$-factors for arbitrary $\pi$, we may assume that $\pi_0$, and hence $\pi$, is generic with respect to the characters obtained from $\psi$, so that Proposition~\ref{lcmultiplicativity} applies. We argue by induction on $m_R$. Writing \eqref{recursivedef} for $({\bf G},{\bf M},\pi)$ and for each $({\bf G}_j,{\bf M}_j,\pi_j)$, Proposition~\ref{lcmultiplicativity} gives
\begin{equation*}
   \prod_{i} \lambda_{r_i}(\psi,w_0)^{-1} \gamma(is,\pi,r_i,\psi) = \prod_j \prod_l \lambda_{r^{(j)}_l}(\psi,w_j)^{-1} \gamma(l a_j s,\pi_j,r^{(j)}_l,\psi).
\end{equation*}
Each $\beta \in \Sigma_{\rm red}(\emptyset,w_0)$ lies in exactly one block, with the same index $i_\beta$ for $({\bf G},{\bf M})$ as for the corresponding pair $({\bf G}_j,{\bf M}_j)$ up to the factor $a_j$, so that, with $b_j = b_j(r_i) = i/a_j$,
\[ \lambda_{r_i}(\psi,w_0) = \prod_{j \in \mathfrak{S}_{r_i}} \lambda_{r^{(j)}_{b_j}}(\psi,w_j). \]
Grouping the right hand side according to $i = l a_j$, we obtain
\begin{equation}\label{prodmult}
   \prod_{i=1}^{m_R} \gamma(is,\pi,r_i,\psi) = \prod_{i=1}^{m_R} \prod_{j \in \mathfrak{S}_{r_i}} \gamma(is,\pi_j,r^{(j)}_{b_j},\psi).
\end{equation}
If $m_R = 1$, this is Property~(iv). Let $m_R > 1$ and $i > 1$. By definition, $\gamma(s,\pi,r_i,\psi) = \gamma(s,\pi,r_1',\psi)$, where $r_1'$ is attached to the pair $({\bf G}_{r_i},{\bf M})$ of Lemma~\ref{induction}, with $m_{R'} < m_R$. The blocks for $({\bf G}_{r_i},{\bf M})$ relative to $\theta_0$ are those of $({\bf G},{\bf M})$ with index in $i\mathbb{Z}$, and the pieces of $({\bf G}_{r_i},{\bf M})$ contributing to $r_1'$ are indexed by $j \in \mathfrak{S}_{r_i}$. For each such $j$, the corresponding pair is, up to an isomorphism as in Lemma~\ref{pinning}, the pair attached by Lemma~\ref{induction} to $({\bf G}_j,{\bf M}_j)$ and $r^{(j)}_{b_j}$, with $b_j = i/a_j$. Indeed, both are obtained from the block $[\beta_j]$ and its multiples with index in $i\mathbb{Z}$, and the two constructions are related by conjugation by Tits representatives, which preserves the pinnings on the simple roots of the Levi subgroups. Hence, the inductive hypothesis applied to $({\bf G}_{r_i},{\bf M})$, together with Lemma~\ref{pinning} and the definition of $\gamma(s,\pi_j,r^{(j)}_{b_j},\psi)$, gives Property~(iv) for $r_i$, $i > 1$. Property~(iv) for $r_1$ now follows from \eqref{prodmult}.

For Property~(v), let $a \in F^\times$ and let $\pi$ be $\psi^a_{\tilde{w}_0}$-generic. Let $\mathfrak{W}^a$ be the system with the same Weyl group element representatives as $\mathfrak{W}$, and Haar measures self-dual with respect to $\psi^a$, see \S~\ref{rankone}. Since the center of $\widetilde{\bf G}$ has trivial $H^1$, there is an $x \in \widetilde{T}$ with $\beta(x) = a$ for every absolute simple root $\beta$; then $\psi^a(n) = \psi(xnx^{-1})$, $n \in U$, so that $\pi_x$ is $\psi_{\tilde{w}_0}$-generic, and $e \in Z_M$ in \eqref{eq:e:element}. We write $e = z(a)$. Proposition~\ref{vary} gives
\begin{equation*}
   C_{\psi^a}(s,\pi,\tilde{w}_0) = a_x(s,\psi^a,\mathfrak{W}^a) C_\psi(s,\pi_x,\tilde{w}_0),
\end{equation*}
and we compute the monomial from \eqref{themonomial}. Here $t = \bar{t} = 1$, and self-dual measures give $b' = |a|_F^{-n/2}$, $n = \dim({}^L\mathfrak{n})$. The roots in $N_{\overline{\theta}}$ are the $-w_l(\gamma)$, $\gamma$ a root in $N_\theta$, and $-w_l$ preserves heights; hence $c(x;N_\theta) = c(x;N_{\overline{\theta}})$ and $\delta_P(e) = c(x;N_\theta)^2$. With $\rho^\vee$ the half sum of positive coroots, this gives $\delta_\alpha(e) = |a|_F^{2 \langle \tilde{\alpha},\rho^\vee \rangle} = |a|_F^{\sum_i i n_i}$, $n_i = \dim({}^L\mathfrak{n}_i)$. Thus
\begin{equation*}
   a_x(s,\psi^a,\mathfrak{W}^a) = \omega_\pi(z(a)) \prod_i |a|_F^{n_i(is-\frac{1}{2})}.
\end{equation*}
Finally, $\lambda(F'/F,\psi^a) = \delta_{F'/F}(a) \lambda(F'/F,\psi)$, with $\delta_{F'/F}$ the discriminant character, so $\lambda_{r_i}(\psi^a,w_0) = \kappa_{r_i}(a) \lambda_{r_i}(\psi,w_0)$ with $\kappa_{r_i}(a) = \pm 1$. Then, arguing recursively with \eqref{recursivedef} and Lemma~\ref{induction}, we obtain the relationship for individual $\gamma$-factors. We observe that in examples, multiplicativity allows us to compute $\omega_\pi(z(a))$ explicitly. Indeed, we can take a principal series representation $\tau$ of $M$ such that $\omega_\tau = \omega_\pi$. 

Furthermore, we obtain individual functional equations for each of the $\gamma$-factors. Namely, this reasoning proves Property~(vi) for $(k,\pi,\psi,S) \in \mathcal{LS}(p)$ with $\psi_{\tilde{w}_0}$-generic $\pi$. Note we define $\gamma$-factors in a way that they are compatible with the functional equation for a $\chi$-generic $\pi$ with the help of Propositions~\ref{vary} and \ref{globalvary}. This completes the proof of Properties~(i)--(vi).

Before continuing, we now establish Property~(ix) for Langlands-Shahidi $\gamma$-factors in our case of positive characteristic.

\begin{lemma}\label{twistgamma}
Let $(F,\pi,\psi) \in \mathfrak{ls}(p)$ and $s_0 \in \mathbb{C}$. Then
\begin{equation*}
   \gamma(s,\delta_\alpha^{s_0} \otimes \pi,r,\psi) = \gamma(s+i_rs_0,\pi,r,\psi).
\end{equation*}
\end{lemma}
\begin{proof}
Assume first that $\pi$ is $\psi_{\tilde{w}_0}$-generic. Then $\delta_\alpha^{s_0} \otimes \pi$ is $\psi_{\tilde{w}_0}$-generic, with the same Whittaker functional, and ${\rm I}(s,\delta_\alpha^{s_0} \otimes \pi) = {\rm I}(s+s_0,\pi)$ on the same space. The integrals defining the intertwining operator and the two Whittaker functionals of \eqref{lslc} are the same for both, since $\tilde{w}_0(\delta_\alpha)^s \, \tilde{w}_0(\delta_\alpha)^{s_0} = \tilde{w}_0(\delta_\alpha)^{s+s_0}$. Hence
\begin{equation*}
   C_\psi(s,\delta_\alpha^{s_0} \otimes \pi,\tilde{w}_0) = C_\psi(s+s_0,\pi,\tilde{w}_0),
\end{equation*}
and, since the $\lambda_{r_i}(\psi,w_0)$ do not depend on $\pi$, equation~\eqref{recursivedef} gives
\begin{equation*}
   \prod_{i=1}^{m_R} \gamma(is,\delta_\alpha^{s_0} \otimes \pi,r_i,\psi) = \prod_{i=1}^{m_R} \gamma(is+is_0,\pi,r_i,\psi).
\end{equation*}
We argue by induction on $m_R$. For $i > 1$, we have $r_i = r_1'$ for the pair $({\bf G}_{r_i},{\bf M})$ of Lemma~\ref{induction}, whose character $\delta_{\alpha}'$ satisfies $\delta_\alpha = (\delta_\alpha')^i$, since $\widetilde{\alpha}$ is replaced by $\widetilde{\alpha}/i$. Thus $\delta_\alpha^{s_0} = (\delta_\alpha')^{is_0}$, and, since $m_{R'} < m_R$, the inductive hypothesis gives $\gamma(s,\delta_\alpha^{s_0} \otimes \pi,r_i,\psi) = \gamma(s+is_0,\pi,r_i,\psi)$. The case $i = 1$ then follows from the previous product formula. For arbitrary $(F,\pi,\psi) \in \mathfrak{ls}(p)$, we use Proposition~\ref{vary}, noting that $(\delta_\alpha^{s_0} \otimes \pi)_x = \delta_\alpha^{s_0} \otimes \pi_x$.
\end{proof}

At this point, we wish to observe that Properties~(i)--(vi), Lemma~\ref{twistgamma}, and the uniqueness of $\gamma$-factors proved below are established independently of \cite{dCaHeLo}. As in characteristic zero \cite{Sh1990}, the construction of $L$-functions and $\varepsilon$-factors for general generic $\pi$ via Langlands' classification requires Property~(vii), Shahidi's tempered $L$-function conjecture; in characteristic $p$ this is Theorem~1.1 of \cite{dCaHeLo}, whose proof uses only Properties~(i)--(vi), Theorem~\ref{lc:rationality} and Proposition~\ref{Lreducibility}. Property~(vii) is thus established in \cite{dCaHeLo}, following the method of Heiermann and Opdam \cite{HeOp2013}, where the authors carefully avoid any circular arguments.

Let us now resume the proof of Theorem~\ref{mainthm}, where we address the construction of $L$-functions and $\varepsilon$-factors, among other aspects. Property~(viii) states the relation connecting Langlands-Shahidi local factors. We show that $\varepsilon$-factors are well defined for tempered representations. Let $(F,\pi,\psi) \in \mathfrak{ls}(p)$ be tempered. Let $P_{\pi,r_i}(z)$ and $P_{\tilde{\pi},r_i}(z)$ be the polynomials of Property~(vii) with $z=q_F^{-s}$ and write
\begin{equation*}
   \gamma(s,\pi,r_i,\psi) \,\ddot{\sim}\, \dfrac{P_{\pi,r_i}(z)}{Q_{\pi,r_i}(z)} \text{ and } \gamma(s,\tilde{\pi},r_i,\psi) \,\ddot{\sim}\, \dfrac{P_{\tilde{\pi},r_i}(z)}{Q_{\tilde{\pi},r_i}(z)},
\end{equation*}
where $\ddot{\sim}$ denotes the equivalence relation of equality up to a monomial in $z$. From Property~(xi), we have
\begin{equation*}
   Q_{\pi,r_i}(z) Q_{\tilde{\pi},r_i}(q_F^{-1}z^{-1}) \,\ddot{\sim}\, P_{\pi,r_i}(z) P_{\tilde{\pi},r_i}(q_F^{-1}z^{-1}).
\end{equation*}
Property~(vii) implies that the Laurent polynomials $P_{\pi,r_i}(z)$ and $P_{\tilde{\pi},r_i}(q_F^{-1}z^{-1})$ have no zeros on $\Re(s) > 0$ and $\Re(s) < 1$, respectively. Then, up to a monomial in $z^{\pm 1}$, we have $L(s,\pi,r_i) = Q_{\tilde{\pi},r_i}(q_F^{-1}z^{-1})^{-1}$ and $L(1-s,\tilde{\pi},r_i) = Q_{\pi,r_i}(z)^{-1}$. Hence, $\varepsilon(s,\pi,r_i,\psi)$ is a monomial in $q_F^{-s}$.

The $\gamma$-part of Property~(ix) is Lemma~\ref{twistgamma}, and its $L$- and $\varepsilon$-parts follow from it: for tempered $\pi$ they follow from the definitions in Properties~(vii) and (viii), and for essentially tempered representations $\delta_\alpha^{s_0} \otimes \pi$ they are used as the definition, as in \cite{Sh1990}; see also \S~4 of \cite{dCaHeLo}.
To proceed to the general $(F,\pi,\psi) \in \mathfrak{ls}(p)$, we use Langlands' classification. More precisely, $\pi$ is a representation of $M_\theta$ and we let $\theta_0 \subset \theta$. Let $\pi_0$ be a tempered generic representation of $M_0 = M_{\theta_0}$ and $\chi$ a character of $A_0 = A_{\theta_0}$ which is in the Langlands situation \cite{BoWa,Si1978}. Then $\pi$ is the Langlands quotient of the representation
\begin{equation*}
   \xi = {\rm Ind}(\pi_{0,\chi}),
\end{equation*}
with $\pi_{0,\chi} = \pi_0 \cdot \chi$.

From Proposition~\ref{lcmultiplicativity} we obtain $(F,\pi_j,\psi) \in \mathfrak{ls}(p,{\bf G}_j,{\bf M}_j)$. Each $\pi_j$ is quasi-tempered. Property~(ix) allows us to define $L(s,\pi_j,r^{(j)}_{b_j})$ and $\varepsilon(s,\pi_j,r^{(j)}_{b_j},\psi)$. We then let
\begin{align*}
   L(s,\pi,r) &= \prod_{j \in \mathfrak{S}_r} L(s,\pi_j,r^{(j)}_{b_j}), \\
   \varepsilon(s,\pi,r,\psi) &= \prod_{j \in \mathfrak{S}_r} \varepsilon(s,\pi_j,r^{(j)}_{b_j},\psi)
\end{align*}
be the definition of $L$-functions\footnote{For $(F,\pi,\psi) \in \mathfrak{ls}(p)$ with $\pi$ unramified as in \S~\ref{partialL}, the $L$-functions $L(s,\pi,r)$ thus constructed agree with \eqref{eq:L:nr:def}, see \cite{dCaHeLo}.} and root numbers. This concludes the existence part of Theorem~\ref{mainthm}.

For uniqueness, we start with a local triple $(F,\pi_0,\psi_0) \in \mathfrak{ls}(p)$. First, for $\gamma$-factors, we note that Property~(iv) allows us to reduce to supercuspidal triples. There is a global function field $k$ with a place $v_0$ such that $k_{v_0} \cong F$, and a quasi-split pair $({\bf G},{\bf M})$ over $k$ giving the local pair at $v_0$, see Lemma~3.2 of \cite{GaLoJEMS}; by Property~(i) we identify $k_{v_0}$ with $F$. Fix a non-trivial character $\psi = \otimes_v \psi_v$ of $k \backslash \mathbb{A}_k$. By Property~(v) we may replace $\psi_0$ by $\psi_{v_0}$. We then globalize $\pi_0$ into a globally generic representation $\pi = \otimes_v \pi_v$ via Lemmas~\ref{hllemma} and \ref{generic:open}.
Notice that partial $L$-functions are uniquely determined. Hence, the functional equation gives a uniquely determined product
\begin{equation*}
   \prod_{v \in S} \gamma(s,\pi_v,r_v,\psi_v).
\end{equation*}
For each $v \in S - \left\{ v_0, v_\infty \right\}$ we can use Proposition~\ref{unrgammaprop}.

At the place $v_\infty$, we still obtain a generic constituent of a tamely ramified principal series
\begin{equation*}
   \pi_{v_\infty} \hookrightarrow {\rm Ind}(\chi_\infty),
\end{equation*}
with $\chi_\infty$ a tamely ramified character of ${\bf T}(k_{v_\infty})$. Property~(iv) gives
\begin{equation}\label{pstorankone}
   \gamma(s,\pi_{v_\infty},r_{v_\infty},\psi_{v_\infty}) = \prod_{j \in \mathfrak{S}_r} \gamma(s,\pi_{j,v_\infty},r^{(j)}_{b_j,v_\infty},\psi_{v_\infty}).
\end{equation}
Each $\gamma$-factor on the right hand side of the product is then obtained from (i) or (ii) of Proposition~\ref{rankoneprop}. The resulting abelian $\gamma$-factors are known to satisfy a functional equation as in Tate's thesis. Equality of Tate $\gamma$-factors with those of Artin on the Galois side can be seen as the analytic version of class field theory \cite{ArTa}. In simpler terms, at $v_\infty$ equation~\eqref{pstorankone} above reduces $\gamma(s,\pi_{v_\infty},r_{v_\infty},\psi_{v_\infty})$ to a product of uniquely determined abelian $\gamma$-factors.

Finally, tempered $L$-functions and $\varepsilon$-factors are uniquely determined by Properties~(vii) and (viii), and then, in general, by Properties~(ix) and (x). \qed

\subsection{Local reducibility properties}\label{BAcomplimentary} The following result combines Lemma~7.5, Proposition~7.3 and Corollary~7.6 of \cite{Sh1990}. Shahidi's proofs of Propositions~7.2 and 7.3, Lemma~7.5 and Corollary~7.6 of [\emph{loc.\,cit.}] carry over to characteristic $p$, using only Properties~(i)--(vi) of Theorem~\ref{mainthm}, multiplicativity being essential, together with the results of Harish-Chandra and Silberger on the Plancherel measure, which hold in all characteristics \cite{Si1979}.

\begin{proposition}\label{Lreducibility}
Let $(F,\pi,\psi) \in \mathfrak{ls}(p)$ be unitary supercuspidal. Then, for each $i$, $L(s,\pi,r_i)^{-1}$ is a, possibly empty, product of factors $1 - a_j q_F^{-s}$ with $|a_j| = 1$; in particular, $L(s,\pi,r_i)$ is holomorphic for $\Re(s) > 0$. If $i>2$, then $L(s,\pi,r_i) = 1$. Also, the following are equivalent:
   \begin{itemize}
      \item[(i)] The product $L(s,\pi,r_1) L(2s,\pi,r_2)$ has a pole at $s=0$.
      \item[(ii)] For one and only one $i=1$ or $i=2$, the $L$-function $L(s,\pi,r_i)$ has a pole at $s=0$.
      \item[(iii)] The representation ${\rm Ind}(\pi)$ is irreducible, $\tilde{w}_0(M) = M$ and $\tilde{w}_0(\pi) \cong \pi$.
   \end{itemize}
\end{proposition}

Also a consequence of the Langlands-Shahidi theory of $L$-functions is Shahidi's result on complementary series. Namely, Theorem~8.1 of \cite{Sh1990}, whose proof carries through in the characteristic $p$ case. The conventions made in \S~\ref{ss:local:not} concerning ${\rm I}(s,\pi)$ agree with those of Shahidi, and we write ${\rm I}(\pi)$ for ${\rm I}(0,\pi)$ as well.

\begin{theorem}\label{complimentary}
   Let $(F,\pi,\psi) \in \mathfrak{ls}(p)$ be unitary supercuspidal. With the equivalent conditions of the previous proposition, choose $i=1$ or $i=2$ such that $L(s,\pi,r_i)$ has a pole at $s=0$. Then:
\begin{itemize}
   \item[(i)] For $0 < s < 1/i$, the representation ${\rm I}(s,\pi)$ is irreducible and in the complementary series.
   \item[(ii)] The representation ${\rm I}(1/i,\pi)$ is reducible with a unique generic subrepresentation which is in the discrete series. Its Langlands quotient is never generic. It is a pre-unitary non-tempered representation.
   \item[(iii)] For $s > 1/i$, the representations ${\rm I}(s,\pi)$ are always irreducible and never in the complementary series.
\end{itemize}
If $\tilde{w}_0(\pi) \cong \pi$ and ${\rm I}(\pi)$ is reducible, then no ${\rm I}(s,\pi)$, $s > 0$, is pre-unitary; they are all irreducible.
\end{theorem}

We refer to \cite{GaLoJEMS} for a generalization of parts of this theorem to discrete series representations. In particular, we prove the basic assumption (BA) of M\oe glin and Tadi\'c \cite{MoTa2002} when ${\rm char}(F) =p$ in \S\S~7.8 and 7.9 of \cite{GaLoJEMS}.

\subsection{The Ramanujan Conjecture implies the Riemann Hypothesis} Let us recall the generalized Ramanujan Conjecture for a quasi-split connected reductive group scheme $\bf G$ over a function field $k$.

\begin{conjecture*}[Ramanujan]
Let $\pi = \otimes' \pi_v$ be a globally generic cuspidal automorphic representation of ${\bf G}(\mathbb{A}_k)$. Then every $\pi_v$ is tempered. Whenever $\pi_v$ is unramified, its Satake parameters satisfy
\begin{equation*}
   \left| \alpha_{j,v} \right| = 1.
\end{equation*}
\end{conjecture*}

Thanks to the work of L. Lafforgue \cite{LaL}, the Ramanujan Conjecture is a theorem for ${\rm GL}_n$. And, we establish the Ramanujan Conjecture for the classical groups in \cite{Lo2009,Lo2019}; see \cite{OpenAI2026} for the split exceptional groups.

For a general connected reductive group, we now express the results of V. Lafforgue, in combination with those of L. Lafforgue, in such a way that we obtain a Langlands functorial lift in the context of the LS method, Theorem~\ref{RamRie:thm}. The main application is to the Riemann Hypothesis. We begin with a preliminary lift.

\begin{proposition}\label{weak:LVlift}
Let $(k,\pi,\psi,S) \in \mathcal{LS}(p)$ and let $r$ be a Langlands-Shahidi representation of dimension $N$. Then there exists an automorphic representation $\mathcal{T}$ of ${\rm GL}_N(\mathbb{A}_k)$ such that
\begin{equation*}
   L^S(s,\pi,r) = L^S(s,\mathcal{T})
\end{equation*}
and for every place $v$ of $k$ we have that
\begin{equation*}
   \gamma(s,\pi_v,r_{v},\psi_v) = \gamma(s,\mathcal{T}_v,\psi_v).
\end{equation*}
\end{proposition}

\begin{proof} We fix a prime $\ell \neq p$, an algebraic closure $\overline{\mathbb{Q}}_\ell$ of the $\ell$-adic numbers and an isomorphism $\iota : \overline{\mathbb{Q}}_\ell \rightarrow \mathbb{C}$. Then $\iota$ induces an isomorphism ${}^LM(\overline{\mathbb{Q}}_\ell) \rightarrow {}^LM({\mathbb{C}})$. We also fix a separable closure $\bar{k}$ of $k$. Then, via the work of V. Lafforgue \cite{LaV}, the cuspidal representation $\pi$ of ${\bf M}(\mathbb{A}_k)$ gives a parameter $\Sigma : {\rm Gal}(\bar{k}/k) \rightarrow {}^LM(\overline{\mathbb{Q}}_\ell)$, such that $\Sigma_v$ agrees with $\pi_v$ via the Satake correspondence at every $v \notin S$. Furthermore, there is a map $\rho_Z : {}^LM \rightarrow {}^LM / {\rm Der}({}^LM) \simeq {}^LZ_M$, where ${}^LZ_M$ is the center of ${}^LM$. The central character $\omega_\pi$ of $\pi$ corresponds to $\rho_Z \circ \Sigma$ under the global Langlands correspondence for tori.

Let $\chi : k^\times \backslash \mathbb{A}_k^\times \rightarrow \mathbb{C}^\times$ be a Gr\"o\ss encharakter, which we identify with a character of ${\rm Gal}(\bar{k}/k)$ via global class field theory. We let $\pi \otimes \chi = \pi \otimes (\chi \circ \delta_0)$, where $\delta_0$ is the generator of ${\rm Hom}_k({\bf M}/{\bf Z}_{\bf G}^\circ,{\rm GL}_1)$ of which $\delta$ is a positive multiple, and $\Sigma \otimes \chi = (\phi_{\delta_0} \circ \chi) \cdot \Sigma$, where $\phi_{\delta_0} : \mathbb{C}^\times \rightarrow Z({}^LM^\circ)$ is the dual cocharacter; see \S~5.2 of \cite{GaLoJEMS}. Semisimplification of $r \circ (\Sigma \otimes \chi)$ gives a semisimple representation
\begin{equation*}
   \mathcal{R}_\chi : {\rm Gal}(\bar{k}/k) \rightarrow {\rm GL}_N(\overline{\mathbb{Q}}_\ell),
\end{equation*}
with the same Artin $L$-functions. We then have $\mathcal{R}_\chi = \rho_1 \oplus \cdots \oplus \rho_d$, with each $\rho_j$ irreducible $N_j$-dimensional. From the work of L. Lafforgue \cite{LaL}, each $\rho_j$ corresponds to a cuspidal automorphic representation of ${\rm GL}_{N_j}(\mathbb{A}_k)$, and $\mathcal{R}_\chi$ corresponds to their isobaric sum $\mathcal{T}_\chi$, an automorphic representation of ${\rm GL}_N(\mathbb{A}_k)$.

Recall that $\pi_v$ is unramified for $v \notin S$. The construction is such that $\mathcal{T}_\chi$ agrees with the Satake correspondence at every $v \notin S$. Then, we have equality of partial $L$-functions
\begin{equation*}
   L^S(s,\pi \otimes \chi,r) = L^S(s,\mathcal{T}_\chi).
\end{equation*}
In particular, we have equality for $\chi$ trivial.

Fix a place $v_0 \in S$. From the Grunwald-Wang theorem of class field theory, there is a Gr\"o\ss encharakter $\eta : k^\times \backslash \mathbb{A}_k^\times \rightarrow \mathbb{C}^\times$ such that $\eta_{v_0}=1$ and $\eta_v$ is highly ramified for $v \in S' = S \setminus \left\{ v_0 \right\}$. At unramified places, we use the compatibility of V. Lafforgue's lift with the Satake correspondence. At places $v \in S'$ we have stability of $\gamma$-factors. Indeed, as in the proof of Theorem~5.1 and Corollary~5.6 of \cite{GaLoJEMS}, if we take $\eta_v$, for $v \in S'$, to be sufficiently highly ramified, we have
\begin{equation*}
   \gamma(s,\pi_v \otimes \eta_v,r_v,\psi_v) = \gamma(s,\mathcal{T}_{\eta,v},\psi_v).
\end{equation*}
Hence, we can conclude equality at $v_0$ by comparing functional equations.
\end{proof}

\begin{theorem}\label{RamRie:thm}
Let $(k,\pi,\psi,S) \in \mathcal{LS}(p)$, let $r$ be a Langlands-Shahidi representation of dimension $N$, and assume that $\pi$ satisfies the Ramanujan Conjecture. Then there is an automorphic representation of ${\rm GL}_N(\mathbb{A}_k)$
\begin{equation}\label{eq:isobaric}
   \mathcal{T} = \mathcal{T}_1 \boxplus \cdots \boxplus \mathcal{T}_d,
\end{equation}
written as an isobaric sum of unitary cuspidal automorphic representations $\mathcal{T}_j$ of ${\rm GL}_{N_j}(\mathbb{A}_k)$, such that
\begin{equation*}
   L^S(s,\pi,r) = L^S(s,\mathcal{T})
\end{equation*}
and at every place $v$ of $k$ we have that
\begin{align*}
   L(s,\pi_v,r_v) &= L(s,\mathcal{T}_v), \\
   \varepsilon(s,\pi_v,r_{v},\psi_v) &= \varepsilon(s,\mathcal{T}_v,\psi_v), \\
   \gamma(s,\pi_v,r_{v},\psi_v) &= \gamma(s,\mathcal{T}_v,\psi_v).
\end{align*}
Furthermore, the zeros of $L(s,\pi,r)$ are contained in the line $\Re(s) = 1/2$.
\end{theorem}
\begin{proof}
Let $\mathcal{T}$ be the functorial lift of Proposition~\ref{weak:LVlift}. From the classification of representations for ${\rm GL}_N$ \cite{JaSh1981}, we can write $\mathcal{T}$ as an isobaric sum
\begin{equation}\label{eq:isobaric:pf}
   \mathcal{T} = \mathcal{T}_1 \boxplus \cdots \boxplus \mathcal{T}_d,
\end{equation}
where $\mathcal{T}_j = \mathcal{T}_{j,0} \nu^{t_j}$ and $\mathcal{T}_{j,0}$ is a cuspidal unitary representation of ${\rm GL}_{N_j}(\mathbb{A}_k)$. The Langlands parameters can be arranged such that $t_1 \leq \cdots \leq t_d$. By construction, $\mathcal{T}$ must correspond to
\begin{equation*}
   \mathcal{R} = \rho_1 \oplus \cdots \oplus \rho_d,
\end{equation*}
where the $\rho_j$'s are, after rearrangement if necessary, the representations of the proof of Proposition~\ref{weak:LVlift}.

Let $v \notin S$. As at the beginning of \S~\ref{partialL}, we have $\pi_v = \pi(\chi_v)$ with Langlands parameter $\iota_M \circ \sigma(\chi_v)$, and by construction $\mathcal{T}_v$ is unramified with Satake parameter the semisimple conjugacy class of $r(\iota_M \circ \sigma(\chi_v))({\rm Fr})$ in ${\rm GL}_N(\mathbb{C})$. Since $\pi_v$ is tempered, all of its eigenvalues have absolute value one. On the other hand, for almost every such $v$ each $\mathcal{T}_{j,v}$ is unramified, and its Satake parameters are $q_v^{-t_j}$ times those of $\mathcal{T}_{j,0,v}$, which have absolute value one, since $\mathcal{T}_{j,0}$ is cuspidal unitary and satisfies the Ramanujan Conjecture by L. Lafforgue \cite{LaL}. The Satake parameter of $\mathcal{T}_v$ is the union of those of the $\mathcal{T}_{j,v}$, hence $t_j = 0$ for each $j$.

Therefore, the isobaric sum of \eqref{eq:isobaric:pf} consists only of cuspidal unitary representations. Furthermore, each $L$-function $L(s,\mathcal{T}_j)$ satisfies the Riemann Hypothesis, again thanks to L. Lafforgue \cite{LaL}. We now compare local $L$-functions at every place $v$ of $k$. Since $\pi$ satisfies the Ramanujan Conjecture, each $\pi_v$ is tempered; and each $\mathcal{T}_v$ is tempered, being an isobaric sum of the local components of the unitary cuspidal $\mathcal{T}_j$, which satisfy the Ramanujan Conjecture by \cite{LaL}. By Property~(vii) of Theorem~\ref{mainthm}, the $L$-function of a tempered representation is determined by its $\gamma$-factor, and the same holds for tempered representations of ${\rm GL}_N(k_v)$. Since $\gamma(s,\pi_v,r_v,\psi_v) = \gamma(s,\mathcal{T}_v,\psi_v)$ for every $v$, by Proposition~\ref{weak:LVlift}, we obtain
\begin{equation*}
   L(s,\pi_v,r_v) = L(s,\mathcal{T}_v)
\end{equation*}
at every place $v$ of $k$. In particular, under the Ramanujan Conjecture the lift $\mathcal{T}$ of Proposition~\ref{weak:LVlift} is compatible with $\gamma$-factors, $L$-functions and $\varepsilon$-factors at every place of $k$. Hence
\begin{equation*}
   L(s,\pi,r) = \prod_v L(s,\pi_v,r_v) = \prod_{j=1}^d L(s,\mathcal{T}_j)
\end{equation*}
satisfies the Riemann Hypothesis.
\end{proof}

\end{document}